\documentclass[10pt,a4paper]{amsart}
\usepackage[english]{babel}
\usepackage{amssymb,xypic,amscd,amsthm, mathtools, amsmath,xypic}

\CompileMatrices
\usepackage[T1]{fontenc}
\newtheorem{defn}{Definition}
\newtheorem{thm}[defn]{Theorem}
\newtheorem{cor}[defn]{Corollary}
\newtheorem{lem}[defn]{Lemma}
\newtheorem{prop}[defn]{Proposition}

\theoremstyle{plain}
\newtheorem{rem}[defn]{Remark}
\theoremstyle{remark}
\newtheorem{exam}{Example}
\numberwithin{equation}{section} \numberwithin{defn}{section}
\usepackage{multicol}
\usepackage{float}
\usepackage{xfrac}

\usepackage{tikz-cd}

\newcommand{\Mat}{\operatorname{Mat}}
\newcommand\ed{\operatorname{End}}
\newcommand{\Img}{\operatorname{Im}}

\newcommand\Ker{\operatorname{Ker}}
\newcommand\aut{\operatorname{Aut}}

\newcommand{\f}{\varphi}
\newcommand{\g}{\psi}
\newcommand{\R}{\mathbb{R}}
\newcommand{\N}{\mathbb{N}}

\newcommand{\C}{\mathbb{C}}

\newcommand\Id{\operatorname{Id}}
\begin{document}

\title[On binary relations defined by GD1 and 1GD inverses]{On the binary relations defined using GD1 and 1GD inverses over infinite dimensional vector spaces}
\author{ Diego Alba Alonso* \\Javier Sánchez González**}

\address{Departamento de Matem\'aticas, ETSII, Universidad de Castilla-La Mancha, 13071 Ciudad Real, Spain}
 \email{ (*)Corresponding Author: daa29@usal.es ; Diego.Alba@uclm.es ORCID:0000-0001-7147-8368}
 \email{ (**) Javier.SGonzalez@uclm.es }

\thanks{This work was supported by {\it Agencia Estatal de Investigación} (Spain) through grant PID2023-151823NB-I00.}

\begin{abstract}
The purpose of this article is to study certain binary relations of endomorphisms over infinite dimensional vector spaces defined by GD1 and 1GD generalized inverses. In order to do so, these generalized inverses are studied over arbitrary vector spaces (namely, infinite dimensional ones) using finite potent endomorphisms. We characterize them in terms of the AST decomposition of a finite potent endomorphism and we obtain algorithms for their respective computation. This theory is then used to characterize the GD1 and 1GD binary relations for finite potent endomorphisms in terms of the AST decomposition and to prove that they define partial orders in the set of finite potent endomorphisms, thus, completing the theory of these generalized inverses for matrices.
\end{abstract}

\maketitle

\bigskip

%%%%%%%%%%%%%%%%%%%%%%%%%%%%%%%%%%%%%%%%%%%%%%%%%%%%%%%%%%%%%
\setcounter{tocdepth}1

\tableofcontents
%%%%%%%%%%%%%%%%%%%%%%%%%%%%%%%%%%%%%%%%%%%%%%%%%%%%%%%%%%%%%
\bigskip

\medskip

\bigskip 


\section{Introduction.}

Let $\Mat_{n\times s}(k)$ denote the ring of $(n \times s)$ matrices over an arbitrary field $k.$ Given $A\in \Mat_{n\times s}(k),$ the index of $A,$ $i(A),$ is the smallest positive integer $m$ such that $\mathrm{rk (A^m)}=\mathrm{rk(A^{m+1})}.$ We say that a matrix $X\in \Mat_{s\times n}(k)$ is an inner inverse or $\{1\}-$inverse of $A$ when $A\cdot X\cdot A=A.$ These are a simple example of the so called generalized inverses of matrix $A.$ In \cite{GDraz}, G-Drazin inverses of a matrix with entries in the complex numbers $A\in \Mat_{n\times n}(\C)$ were introduced as the matrices $X\in \Mat_{n\times n}(\C)$ such that $$A\cdot X \cdot A=A \quad \& \quad A\cdot X\cdot A^{m+1}=A^m \quad \& \quad A^{m+1}\cdot X=A^m, $$ where $m=i(A).$ It is customary to denote the set of $\{1\}-$inverses of a matrix $A$ by $A(1)$, as well as using the notation $A^-\in A(1)$ for its elements (although they are not unique). Similarly, $A(GD)$ denotes the set of G-Drazin inverses of a matrix $A$ and any of its elements are denoted by $A^{GD}$ (even though they are not unique). These generalized inverses have been generalized to the context of Banach spaces, see \cite{GMos}. Further, in \cite{GD1N}, a new generalized inverse was studied combining $\{1\}-$inverses and G-Drazin inverses. Precisely, the composition $A^{GD} \cdot A \cdot A^-$ for certain $A^- \in A(1)$ and $A^{GD}\in A(GD)$ was called as GD1 inverse of $A.$ Different GD1 inverses appear varying the $\{1\}-$inverse and the G-Drazin inverse considered. Simultaneously, 1GD inverses were introduced as the ``duals'' of GD1 inverse, this is, as the composition $A^-\cdot A \cdot A^{GD}.$ Again, these generalized inverses were studied in the framework of Banach space operators in \cite{sahoo}. It is worth pointing out that in this work, GD1 and 1GD inverses were investigated using the core-quasinilpotent decomposition as well as the closed range decomposition operator. Moreover, the interconnection with the Drazin inverse (in whatever way one could define it, as we will not use it in this paper) and other properties considering idempotent conditions and projectors were analyzed. Our treatment of this topic is completely different from this point of view.\\In both works, \cite{GD1N} and \cite{sahoo}, a binary relation dealing with GD1 inverses (and 1GD inverses respectively) was introduced. In \cite[Theorem 2.14, Theorem 3.14]{GD1N}, these relations were proved to be equivalent to other known relations, giving as an immediate corollary that they were partial orders in the set of matrices of index lesser or equal than 1.\\
In \cite{Ta}, John Tate introduced the concept of finite potent endomorphisms. Let $V$ denote an arbitrary $k-$vector space, in general, infinite-dimensional. The ring of endomorphisms over the vector space $V$ will be denoted as $\ed_k(V).$ Given $\f \in \ed_k(V)$ we will say that $\f$ is finite potent when $\f^n (V)$ is a finite dimensional $k-$vector subspace of $V,$ where $\f^n=\f\circ \dots \circ \f$ composed $n$ times. Further, in \cite{AST}, these endomorphisms were characterized in terms of a $\f -$invariant decomposition of the vector space which will be called AST decomposition.\\
This article is a contribution to the theory of Generalized Inverses and Matrix Partial Orders. Namely, we generalize the theory of GD1 inverses and 1GD inverses to arbitrary vector spaces, in general, infinite dimensional ones, using the theory of finite potent endomorphisms. We extend the studied binary relations defined by these generalized inverses and we prove that they are partial orders in the set of finite potent endomorphisms over arbitrary vector spaces.\\ The outline of this work, together with the results obtained, can be summarized as follows: \begin{itemize}
\item Generalization of GD1 (respectively 1GD) inverses to infinite dimensional vector spaces (without inner product, nor Banach or Hilbert spaces) in Section \ref{ss: 1GD fp} (respectively Section \ref{ss: 1GD fp}). A characterization of these generalized inverses using the AST decomposition, Theorem \ref{T: Caract GD1 INV AST} (respectively Theorem \ref{T: Charact 1GD AST e inva}). Moreover, study of the structure of the set of all GD1 inverses of a finite potent endomorphism (respectively 1GD), Theorem \ref{T: EstructuraGD1} (Theorem \ref{T: Estructura1GD}).

\item Application of the previous theory to Matrix Theory. Bearing in mind the well known relationship between endomorphisms over finite dimensional vector spaces and finite square matrices, the previous theory is used to study the set of GD1 inverses of a finite square matrix over an arbitrary field (respectively the set of 1GD inverses of a finite square matrix). Namely, if $A=A_1+A_2$ is the core-nilpotent decomposition of $A$ (Section \ref{ss: CNdecomp01010101} and $\mathrm{N_u(A)}$ denotes the nullspace of $A,$ then by Theorem \ref{T: Char AGD1 fin} and Theorem \ref{T: Char A1GD fin}, we have bijections: $$\begin{aligned}&A(GD1) \simeq  k^{[\mathrm{dim } N_u(A)]\cdot (\mathrm{rk (A)}+\mathrm{rk (A_2)}) }\\&A(1GD) \simeq  k^{[\mathrm{dim }N_u(A)]\cdot (\mathrm{rk (A)}+\mathrm{rk (A_2)}) } .\end{aligned} $$ Moreover, the algorithms for the explicit calculation of these sets are offered in Section \ref{ss: Computation GD1 Matrices} and Section \ref{ss: Computation 1GD Matrices}, as well as some illustrative examples, in Section \ref{ss. Examples Algorithms}.

\item In Section \ref{ss: GD1o y 1GDo} the binary relations defined by GD1 and 1GD inverses in \cite{GD1N} are extended and studied over infinite dimensional vector spaces. After characterizing them using the AST decomposition of the operators involved, in Theorem \ref{T: Charact GD1 Order FP} and Theorem \ref{T: 1GD Orden Charact}, we prove that they are both partial orders in the set of finite potent endomorphisms (Theorems \ref{T: GD1 Partial Order} and \ref{T: 1GD Orden Parcial}). This completes the theory of these binary relations, as in the matricial case presented in \cite{GD1N}, the authors prove that these relations are equivalent to other known ones defining partial orders (see \cite[Theorem 2.14, Theorem 3.14]{GD1N}), but only in set of matrices of index lesser or equal than 1.\\Moreover, we give a non trivial example (Example \ref{Ex: Non Trivial}) of matrices ordered for these relations which are of index different than 1, using the characterizations that we presented previously.

\item As a side note we show that there is no relation between the binary relations involving GD1 and 1GD inverses and the subsets of GD1 and 1GD inverses of the linear operators (or matrices) defining this relation, Remark \ref{R: Sin Contencion}.

\item To conclude, we present two ``new'' binary relations using compositions of GD1 and 1GD inverses (Definitions \ref{D: GD11GD bin rel} and \ref{D: 1GDGD1 bin rel}) and we show that they are equivalent to the minus partial order and the G-Drazin partial order respectively (Theorem \ref{T: Bilateral Relations are GDraz and Minus}).

\end{itemize}

We shall point out that as far as the authors know, these results and the approach here presented are not stated previously in the literature. Moreover, we highlight that every proof and result exposed can be specialized to finite square matrices over arbitrary ground fields. Finite potent endomorphisms do not form an ideal of the endomorphisms, namely, the sum and the composition of two finite potent endomorphisms is not, in general, a finite potent endomorphism; and therefore, the generalization presented here is not merely a generalization from finite dimensional vector spaces (finite square matrices) to infinite dimensional vector spaces, but one that also deals with additional problems derived from the impossibility to use the ordinary ring structure of endomorphisms.

\section{Preliminaries} \label{S:pre}
This section is included for the sake of completeness, to fix notations and to recall results that will be of importance later on.
\subsubsection{Finite Potent endomorphisms.}\label{ss: FP endos}

In \cite{Ta}, John Tate introduced the concept of finite potent endomorphisms. Let $V$ denote an arbitrary $k-$vector space, in general, an infinite dimensional one. The ring of endomorphisms over the vector space $V$ will be denoted as $\ed_k(V).$ Given $\f \in \ed_k(V)$ we will say that $\f$ is finite potent when $\f^n (V)$ is a finite dimensional $k-$vector subspace of $V,$ we shall note that $\f^n=\f\circ \dots \circ \f$ n times. 

 In 2007, M. Argerami, F. Szechtman and R. Tifenbach showed in \cite{AST} that an endomorphism $\varphi$ is
finite potent if and only if $V$ admits a $\varphi$-invariant
decomposition $V = U_\varphi \oplus W_\varphi$ such that
$\varphi_{\vert_{U_\varphi}}$ is nilpotent, $W_\varphi$ is finite
dimensional and $\varphi_{\vert_{W_\varphi}} \colon W_\varphi
\overset \sim \longrightarrow W_\varphi$ is an isomorphism.
This decomposition is unique and we shall call it \textit{the $\varphi$-invariant AST-decomposition of $V$.}

Moreover,  we shall call ``index of $\varphi$'', $i(\varphi)$, to the nilpotency order of $\varphi_{\vert_{U_\varphi}}$.  One has that $i(\varphi) = 0$ if and only if $V$ is a finite-dimensional vector space and $\varphi$ is an automorphism.

We shall remark that the sum and the composition of finite potent endomorphism is not necessarily a finite potent endomorphism as can be seen by the following example. Let us consider the $k-$vector space  $V=\underset{i\in \N}{\oplus}<v_i>.$ Moreover, let us define the following endomorphisms:
$$\f(v_{i}) = \left \{ \begin{array}{ccl}  v_{i+1} & \text{ if } &  i \text{ is odd } \\  0 & \text{ if } & i \text{ is even } \end{array} \right . \text{ and }\quad \g(v_{i}) = \left \{ \begin{array}{ccl} 0 & \text{ if } &  i \text{ is odd }  \\   v_{i-1} & \text{ if } & i \text{ is even } \end{array} \right .$$
Notice that both of them are finite potent endomorphisms, as they are nilpotent. Then:
$$(\f+\g)(v_{i}) = \left \{ \begin{array}{ccl}  v_{i+1} & \text{if} & i\text{  is odd } \\  v_{i-1} & \text{ if } & i\text{ is even } \end{array} \right . \text {and }\quad (\f\circ \g)(v_{i}) = \left \{ \begin{array}{ccl}  0 & \text{if} & i\text{ is odd }  \\  v_i & \text{if} & i\text{ is even }  \end{array} \right .,$$ from where we deduce that the sum and the composition of finite potent endomorphisms is not a finite potent endomorphism.

Basic examples of finite potent endomorphisms are all endomorphisms of a finite-dimensional vector space and finite rank or nilpotent endomorphisms of infinite-dimensional vector spaces.

For more details on the theory of finite potent endomorphisms, the reader is referred to \cite{Pa} and \cite{Pa-CN}.

\subsubsection{CN Decomposition of a Finite Potent Endomorphism}\label{ss: CNdecomp01010101}

Let us start by recalling what the core-nilpotent decomposition of an square matrix is. The index of a matrix $A$ is the smallest positive integer such $m$ that $\mathrm{rk (A^m)}=\mathrm{rk (A^{m+1})},$ where $\mathrm{rk}$ denotes the rank.
\begin{thm}\cite[Theorem 2.2.21]{Ind}
Let $A$ be a $n\times n$ matrix. Then $A$ can be written as the sum of matrices $A_1$ and $A_2$ i.e. $A=A_1+A_2$ where $\mathrm{rk(A_1)}=\mathrm{rk(A_1^2)}$ (i.e $i(A)\leq 1$), $A_2$ is nilpotent and $A_1\cdot A_2=0=A_2\cdot A_1.$
\end{thm}
The generalization of this decomposition to arbitrary vector spaces was presented in \cite{Pa-CN}: 
if $V$ is an arbitrary $k$-vector space, given a finite potent endomorphism $\varphi \in \ed_k (V)$,  there exists a unique decomposition $\varphi = \varphi_{_1} + \varphi_{_2}$, where $\varphi_{_1}, \varphi_{_2} \in \ed_k (V)$ are finite potent endomorphisms satisfying that:

\begin{itemize}

\item $i(\varphi_{_1}) \leq 1$;

\item $\varphi_{_2}$ is nilpotent;

\item $\varphi_{_1} \circ \varphi_{_2} = \varphi_{_2} \circ \varphi_{_1} = 0$.

\end{itemize}

Also, the following hold:  \begin{equation} \label{eq:index1} \varphi = \varphi_1 \Longleftrightarrow U_\varphi = \Ker \varphi \Longleftrightarrow  W_\varphi = \text{ Im } \varphi  \Longleftrightarrow i(\varphi) \leq 1\, .\end{equation}

Moreover, if $V = W_{_\varphi}\oplus U_{_\varphi}$ is the AST-decomposition of $V$ induced by $\varphi$, then $\varphi_{_1}$ and $\varphi_{_2}$ are the unique linear maps such that:

\begin{equation} \label{eq:expl-CN-exp-3498353} \varphi_{_1} (v) = \left \{ \begin{aligned} \varphi (v) \, &\text{ if } \, v\in W_{_\varphi} \\ \, 0 \quad &\text{ if } \, v\in U_{_\varphi} \end{aligned} \right . \quad \text{ and } \quad \varphi_{_2} (v) = \left \{ \begin{aligned} \, 0 \quad &\text{ if } \, v\in W_{_\varphi} \\ \varphi (v) \, &\text{ if } \, v\in U_{_\varphi} \end{aligned} \right . \quad \, .\end{equation}

\medskip

\subsubsection{Jordan Bases of a nilpotent endomorphism} \label{ss:nilpotent-basis}

Henceforth, $V$ will be a vector space over an arbitrary field $k$ and let $g\in\ed_k (V)$ be a nilpotent endomorphism. If $m$ is the nilpotency index of $g$, according to the statements of \cite{Pa},  setting $U_i^g = \Ker g^i/[\Ker g^{i-1} + g(\Ker g^{i+1})]$ with $i\in \{1,2,\dots, m\}$, $\mu_i (V,g) = \text{dim}_k U_i^g$ and $S_{\mu_i (V,g)}$ a set such that $\# S_{\mu_i (V,g)} = \mu_i (V,g)$ with $S_{\mu_i (V,g)} \cap S_{\mu_j (V,g)} = \emptyset$ for all $i \ne j$, one has that there exists a family of vectors $\{{ {v_{s_i}}}\}$  that determines a Jordan basis of $g$: \begin{equation} \label{eq:jordan-basis-B} B = \underset {\begin{aligned} s_i &\in S_{{\mu}_i (V,g)} \\ 1 &\leq i \leq m \end{aligned}} {\bigcup} \{{ {v_{s_i}}}, g ({{ v_{s_i}}}), \dots , g^{i-1} ({ {v_{s_i}}})\}\, .\end{equation}
Moreover, if we write $H_{s_i}^g = \langle { {v_{s_i}}}, g ({{ v_{s_i}}}), \dots , g^{i-1} ({ {v_{s_i}}}) \rangle$, the basis $B$ induces a decomposition \begin{equation} \label{eq:decomp} V =  \underset {\begin{aligned} s_i &\in S_{{\mu}_i (V,g)} \\ 1 &\leq i \leq m \end{aligned}} \bigoplus H_{s_i}^g\, .\end{equation}

\subsubsection{Bases of a Finite Potent endomorphism} \label{ss:basis-fp.endormo}

Let us now consider a finite potent endomorphism $\varphi \in \ed_k (V)$ with CN-decomposition $\varphi = \varphi_{_1} + \varphi_{_2}$ and that induces the AST-decomposition $V = U_\varphi \oplus W_\varphi$. Keeping the above notation, if $m$ is the nilpotency order of $\varphi_2$, we can construct a basis $B_V = B_{W_\varphi} \cup B_{U_\varphi}$ of $V$ where $$ B_{W_\varphi} = \{w_1, \dots, w_r\}$$\noindent is a basis of $W_\varphi$ ($r = \text{dim}_k \, W_\varphi$) and $$ B_{U_\varphi} =  \underset {\begin{aligned} s_i &\in S_{{\mu}_i (U_\varphi, \varphi)} \\ 1 &\leq i \leq m \end{aligned}} {\bigcup} \{{ {v_{s_i}}}, \varphi ({{ v_{s_i}}}), \dots , \varphi^{i-1} ({ {v_{s_i}}})\}$$\noindent is a Jordan basis of $U_\varphi$ determined by $\varphi_{\vert_{U_\varphi}}$.

If $\varphi = \varphi_{_1} + \varphi_{_2}$ is the CN-decomposition of $\varphi$, it is clear that  $$ B_{U_\varphi} =  \underset {\begin{aligned} s_i &\in S_{{\mu}_i (U_\varphi, \varphi)} \\ 1 &\leq i \leq m \end{aligned}} {\bigcup} \{{ {v_{s_i}}}, \varphi_{_2} ({{ v_{s_i}}}), \dots , \varphi_{_2}^{i-1} ({ {v_{s_i}}})\}$$\noindent and 
\begin{equation}\label{eq: KerFP}
\Ker \varphi =  \underset {\begin{aligned} s_i &\in S_{{\mu}_i (U_\varphi, \varphi)} \\ 1 &\leq i \leq m \end{aligned}} {\bigoplus} \langle \varphi^{i-1} ({ {v_{s_i}}}) \rangle =  \underset {\begin{aligned} s_i &\in S_{{\mu}_i (U_\varphi, \varphi)} \\ 1 &\leq i \leq m \end{aligned}} {\bigoplus} \langle \varphi_{_2}^{i-1} ({ {v_{s_i}}}) \rangle\, .
\end{equation}

\subsection{Generalized inverses in arbitrary dimensional vector spaces}\label{ss: GIarbitrarydim}

The purpose of this section is to summarize the theory of generalized inverses on arbitrary vector spaces that has been developed using finite potent endomorphisms.

\subsubsection{\{1\}-inverses and \{1,2\}-inverses of a finite potent endomorphism}\label{ss: 1inv fp}
This section contains some results presented in \cite{Die-Fpa} and\cite{PaAl} dealing with the characterization of $\{1\}-$inverses and reflexive generalized inverses of finite potent endomorphisms.

\begin{prop}\cite[Proposition 3.3]{PaAl}\label{P: 1invfp}
If $\f \in \ed_k(V)$ is a finite potent endomorphism, then an endomorphism $\f^{-}\in \ed_k(V)$ is a $\{1\}-$inverse of $\f$ if and only if $\f^{-}$ satisfies that, maintaining the notations above: \begin{itemize}

\item $\f^{-} (w_h) = (\f_{\vert_{W_\f}})^{-1} (w_h) + u_h$ for each $h\in \{1, \dots, r\}$;

\item $\f^{-} (\f^j ({{ v_{s_i}}})) = \f^{j-1} ({{ v_{s_i}}}) + u_{s_i}^j$ for every $s_i \in S_{{\mu}_i (U_{\f}, \f)}$ and $j\in \{1, \dots, i-1\}$;

\item $\f^{-} ({{ v_{s_i}}}) = {\tilde v}_{s_i}$ for every $ s_i \in S_{{\mu}_i (U_{\f}, \f)}$; \noindent
\end{itemize}
where ${\tilde v}_{s_i} \in V$ and $u_h,  u_{s_i}^j \in \Ker \f$ for each $h\in \{1, \dots, r\}$ and for every $s_i \in S_{{\mu}_i (U_{\f}, \f)}$ and $j\in \{1, \dots, i-1\}$.
\end{prop}

If $C = (c_{ij})$ is the matrix associated to $\varphi_{\vert_{W_\varphi}}$ in the basis $B_{W_\varphi}$, we can write $$\varphi (w_j) = \sum_{i=1}^r c_{ij} w_i$$\noindent for every $j\in \{1, \dots, r\}$.

The explicit characterization of the reflexive generalized inverses of a finite potent endomorphism was proved in the following:

\begin{prop}\cite[Proposition 3.2]{Die-Fpa} \label{P: Char Reflex Fp} Let $V$ be an arbitrary $k$-vector space and let us consider a finite potent endomorphism $\varphi \in \ed_k (V)$. With the previous notation, an endomorphism ${\hat \varphi} \in \ed_k (V)$ is a reflexive generalized inverse if and only if it satisfies the following three conditions \begin{equation} \label{firstldi3yd6d5dh36dd} {\hat \varphi} (w_h) = (\varphi_{\vert_{W_\varphi}})^{-1} (w_h) +  \underset {\begin{aligned} s_{i'} &\in S_{{\mu}_{i'} (U_\varphi, \varphi)} \\ 1 &\leq i' \leq n \end{aligned}} \sum \lambda_{s_{i'}}^h\cdot \varphi^{i'-1} ({ {v_{s_{i'}}}})\, ,\end{equation} \noindent with $\lambda_{s_{i'}}^h\in k$ for each $s_{i'} \in S_{{\mu}_{i'} (U_\varphi, \varphi)}$ and each $h\in \{1, \dots, r\}$ and where only a finite number of the scalars $\{\lambda_{s_{i'}}^h\}$ are different from zero;

\begin{equation} \label{secodn83736yhdudye6} {\hat \varphi} (\varphi^j ({{ v_{s_i}}})) = \varphi^{j-1} ({{ v_{s_i}}}) + \underset {\begin{aligned} s_{i'} &\in S_{{\mu}_{i'} (U_\varphi, \varphi)} \\ 1 &\leq i' \leq n \end{aligned}} \sum \beta_{s_{i'}}^{s_i,j}\cdot \varphi^{i'-1} ({ {v_{s_{i'}}}})\end{equation}\noindent with $\beta_{s_{i'}}^{s_i,j} = 0$ for almost all  $s_{i'} \in S_{{\mu}_i (U_\varphi, \varphi)}$ and $j\in \{1, \dots, i-1\};$
 and \begin{equation} \label{eq:38d76dghydtdg} {\hat \varphi} ({{ v_{s_i}}}) = \sum_{j=1}^r \gamma_j^{s_i}\cdot w_j + \underset {\begin{aligned} s_{i'} &\in S_{{\mu}_{i'} (U_\varphi, \varphi)} \\ 1 &\leq i' \leq n \end{aligned}} \sum [\sum_{l=0}^{i'-1} \xi_{s_{i'}}^{s_i,l} \varphi^{l} ({{v_{s_{i'}}}})]\end{equation} \noindent with \begin{equation} \label{eq:8d7hdyjdikdjnd} \xi_{s_{i'}}^{s_i,i'-1} = \sum_{j,h=1}^r (\lambda_{s_{i'}}^h \cdot c_{hj} \cdot \gamma_j^{s_i}) + \underset {\begin{aligned} s_{i''} &\in S_{{\mu}_{i''} (U_\varphi, \varphi)} \\ 1 &\leq i'' \leq n \end{aligned}} \sum {( \sum_{l=0}^{i''-2} [\xi_{s_{i''}}^{s_i,l} \cdot \beta_{s_{i'}}^{s_{i''},l+1}] )}\, ,\end{equation}\noindent with  $\xi_{s_{i'}}^{s_i,l} = 0$ for almost all  $s_{i'} \in S_{{\mu}_i (U_\varphi, \varphi)}$ and $l\in \{1, \dots, i-1\}$.
\end{prop}

%\subsubsection{Drazin inverse of a finite potent endomorphism}\label{ss: Drazin FP}
%In this short section we recall some of the main properties of the Drazin inverse of a finite potent endomorphism from \cite{DraFP}.
%\begin{thm}\cite[Theorem 3.4]{DraFP}\label{T: Inv Drazin fp}
%For every $\f\in \ed_k^{fp}(V)$ with $i(\f)=m$ there exists a unique finite potent endomorphism, that we will be denoted as $\f^D\in \ed_k^{fp}(V),$ satisfying: \begin{itemize}
%\item $\f^{m+1}\circ \f^D=\f^m;$
%\item $\f^D\circ \f \circ \f^D=\f^D;$
%\item $\f^D\circ \f=\f \circ \f^D.  $
%\end{itemize}
%\end{thm}

%From the proof of this theorem ones gets the explicit expression of the Drazin inverse over the AST decomposition of a finite potent endomorphism, which is:

%\begin{thm}
%Let $\f \in \ed_k^{fp}(V)$ be a finite potent endomorphism and let $V=W_{\f}\oplus U_{\f}$ the AST decomposition it induces. The operator $\f^D$ (Theorem \ref{T: Inv Drazin fp}) is the Drazin inverse of $\f$ if and only if $$\f^D(v) = \left \{ \begin{array}{ccl} (\f_{\vert_{W_{\f}}})^{-1}(v) & \text{ if } & v\in W_{\f} \\  0 & \text{ if } & v\in U_{\f} \end{array} \right . .$$
%\end{thm}

%Some of the main properties of the Drazin inverse of a finite potent endomorphism are: \begin{itemize}
%\item $(\f^D)^D=\f$ if and only if $i(\f)\leq 1;$
%\item $\f^D=\f$ if and only if $\f_{\vert_{U_{\f}}}=0$ and $(\f_{\vert_{W_{\f}}})^2=\Id_{\vert_{W_{\f}}};$
%\item if $\f^2=\f$ then $\f^D=\f.$
%\end{itemize}

\subsubsection{G-Drazin inverses of a finite potent endomorphism}\label{ss: GDrazin fp}

Finally, in this section, some results concerning G-Drazin inverses of finite potent endomorphisms are gathered from \cite{FPGD}.

\begin{defn}\cite[Definition 3.1]{FPGD}\label{D: GDrazin fp}
Given a finite potent endomorphism $\f\in \ed_k(V),$ we say that an endomorphism $\f^{GD}\in \ed_k(V)$ is a G-Drazin inverse of $\f$ when it satisfies that: \begin{align*}
\f\circ \f^{GD}\circ \f &=\f\\
\f^{GD}\circ \f^m&=\f^m\circ \f^{GD}
\end{align*} where $i(\f)=m.$
\end{defn}

From \cite[Proposition 3.4]{FPGD} it follows that:

\begin{prop}\label{P: Charact GDrazinfp}
Let $\f\in \ed_k(V)$ be a finite potent endomorphism of index $i(\f)=m.$ Then, $\f^{GD}\in \ed_k(V)$ is a G-Drazin inverse of $\f$ if and only if $\f^{GD}$ verifies: \begin{itemize}
\item $\f^{GD}(w)=(\f_{_{\vert_{W_{\f}}}})^{-1}(w)$ for any $w\in W_{\f};$
\smallskip
\item $\f^{GD}(\f^j(v_{s_i}))=\f^{j-1}(v_{s_i})+u_{s_i}^j,$ with $u_{s_i}^j \in \Ker(\f)$ for every $s_i\in S_{_{\mu_i}(U_{\f},\f)},$ and with $j\in \{1,\dots, i-1\};$
\smallskip 
\item $\f^{GD}(v_{s_i})=\tilde{v}_{s_i}$ for every $s_i\in S_{_{\mu_i}(U_{\f},\f)};$ 
\end{itemize} where $\tilde{v}_{s_i}\in U_{\f}.$
\end{prop}

\begin{lem}\cite[Lemma 6.2]{RendDaa}\label{P: CompG-Drazin}
Let $\f \in \ed_k(V)$ be a finite potent endomorphism of index $i(\f)=m$ and let $\f^{GD},\tilde{\f}^{GD}\in X_{\f}(GD)$ be two G-Drazin inverses of $\f .$ Then, the composition $\f^{GD}\circ \f \circ \tilde{\f}^{GD}$ is a G-Drazin inverse of $\f .$ 
\end{lem}

\subsection{Pre-orders and partial orders on arbitrary vector spaces}\label{ss: Pre and PO}
In this section we briefly recall some results concerning pre-orders and partial orders on infinite dimensional vector spaces that will be needed in the last section of this article when studying certain binary relations for linear operators on infinite dimensional vector spaces.
\subsubsection{Formalisms of the theory of Matrix Partial Orders}\label{ss: Formalisms}
Let us recall the notion of $\mathcal{G}-$based order relation which can be found in \cite[Chapter 7]{Ind}.
\begin{defn}\cite[Definition 7.2.1]{Ind}
Let $\mathcal{P}(\ed_k(V))$ denote the power set (class of subsets) of $\ed_k(V).$ A g-map is a map of sets $\mathcal{G}\colon \ed_k(V)\to \mathcal{P}(\ed_k(V))$ such that for each $\f \in \ed_k(V),$ $\mathcal{G}(\f)$ is a certain subset (possibly non-empty) of $X_{\f}(1).$
\end{defn}

\begin{defn}\cite[Definition 7.2.3]{Ind}\label{D: G-based rel}
Let $\mathcal{G}\colon \ed_k(V)\to \mathcal{P}(\ed_k(V))$ be a g-map. For $\f, \g \in \ed_k(V),$ we say $\f \, <^{\mathcal{G}} \, \g$ if there exists a $g\in \mathcal{G}(\f)$ such that $\f \circ g=\g \circ g$ and $g\circ \f=g\circ \g .$ The binary relation ``$<^{\mathcal{G}}$'' is called $\mathcal{G}-$based order relation.
\end{defn}

\subsubsection{Space Pre-Order and G-Drazin partial order}\label{ss: Space}

The space pre-order, a basic tool to study binary relations of operators which include $\{1\}-$inverses on their definitions, has been studied in infinite dimensional vector spaces in \cite{RendDaa}. Similarly, the G-Drazin partial order for operators with a notion of index, core-nilpotent endomorphisms, has been studied in the same reference. Every finite potent endomorphism is a core-nilpotent endomorphism, here we will only include the definition of the G-Drazin partial order adapted to our context as we will need no more generality.

\begin{defn}\cite[Definition 3.2]{RendDaa}\label{D: Space Pre-Ord}
Let us consider two linear operators $\f, \g \in \ed_k(V).$ The linear operator $\f$ is said to be below the linear operator $\g$ under the space pre-order if $\Img(\f)\subseteq \Img(\g)$ and $\Ker(\g)\subseteq \Ker(\f).$ When this happens we write $\f \, <^s \, \g .$
\end{defn}

\begin{defn}\cite[Definition 6.1, Corollary 6.3, Theorem 6.14]{RendDaa}\label{D: G-Drazin order fp}
Let $\f , \g \in \ed_k(V)$ be two finite potent endomorphisms. We will say that $\f$ is below $\g$ for the G-Drazin relation when there exists a certain G-Drazin inverse $\f^{GD}$ of $\f,$ such that \begin{align*}
\f \circ \f^{GD}& =\g \circ \f^{GD}\\
\f^{GD}\circ \f & =\f^{GD}\circ \g.
\end{align*}
When this happens we will write $\f \, \leq^{GD}\, \g.$ This relation is a partial order.
\end{defn}

%\begin{thm}\cite[Theorem 3.18]{RendDaa}\label{T: Charact Space}
%Let $\f,\g \in \ed_k(V)$ be two linear operators. The following are equivalent: \begin{itemize}
%\item[I.)]$\f <^s \g,$
%\item[II.)]$\f=\g\g^{-}\f = \f \g^{-}\g$ for a $\g^{-}\in X_{\g}(1),$
%\item[III.)]$\f=\g\g^{-}\f =\f\g^{-}\g$ for any $\g^{-}\in X_{\g}(1).$
%\end{itemize}
%\end{thm}

\subsubsection{Minus partial order of finite potent endomorphisms}\label{ss: MinusFP}

Recently, one of the authors of this work has extended the minus partial order theory to arbitrary vector spaces, in general, infinite dimensional ones, using finite potent endomorphisms in \cite{MinusDaa}. Here we will recall the basic results that will be used later in this paper.\\
Given a finite potent endomorphism $\f$ over an arbitrary $k-$vector space $V,$ we will denote by $X_{\f}^{fp}(1)$ to the set of $\{1\}-$inverses of $\f$ that are finite potent.

\begin{defn}\cite[Definition 4.1]{MinusDaa} \label{D: Minus FP}
Given two finite potent endomorphisms $\f, \g \in \ed_k(V),$ we will say that $\f$ is under $\g$ for the minus partial order of finite potent endomorphisms when $X_{\g}^{fp}(1)\subseteq X_{\f}^{fp}(1).$	 When this happens we write $\f \leq^{-}\g.$ 
\end{defn}

\begin{thm}\cite[Theorem 4.2]{MinusDaa}\label{T: Minus is PO FP}
The minus order of finite potent endomorphisms, ``$\, \leq^{-} \,$'' is a partial order in the set of finite potent endomorphisms.
\end{thm}

\begin{rem}\cite[Remark 4.4]{MinusDaa}\label{R: MinusCoinciden}
The usual minus partial order restricted to the subset of finite potent endomorphisms coincides with the minus partial order for finite potent endomorphisms. This is, given two endomorphisms over an arbitrary k-vector space $f, g \in \ed_k(V)$ then $$f\leq^- g \Longleftrightarrow X_{g}(1)\subseteq X_{f}(1) \Longleftrightarrow X_{g}^{fp}(1)\subseteq X_{f}^{fp}(1). $$
\end{rem}

Now we recall some of the equivalent characterizations of this relation that we will use to study some binary relations defined in this article.

\begin{thm}\cite[Theorem 4.5]{MinusDaa}\label{T: Equiv Minus FP}
Let $\f, \g \in \ed_k(V)$ be finite potent endomorphisms. The following statements are equivalent: \begin{itemize}
\item[I.)] There exists $\f_1^{-},\f^{-}_2 \in X_{\f}(1)$ such that $\f^{-}_1\circ \f = \f^{-}_1\circ \g$ and $\f\circ \f^{-}_2=\g\circ \f^{-}_2.$ 
\item[II.)] There exists $\f^{-}\in X_{\f}(1)$ such that $\f^-\circ \f=\f^-\circ \g$ and $\f \circ \f^-=\g\circ \f^-.$ 
\item [II'.)]There exists $\f^+\in X_{\f}(1,2)$ such that $\f^+\circ \f=\f^+\circ \g$ and $\f \circ \f^+=\g\circ \f^+.$ 
\item[III.)] $X_{\g}^{fp}(1)\subseteq X_{\f}^{fp}(1).$ 
\item[IV.)]For any $\g^-\in X_{\g}^{fp}(1),$ then $\f\circ \g^{-}\circ (\g-\f) =(\g-\f)\circ \g^{-}\circ \f=0$ 
\item[V.)] There exists a $\f^- \in X_{\f}(1)$ such that: $\f=\f\circ\f^-\circ \g=\g\circ \f^-\circ\f.$ 
\item[VI.)] There exist idempotent endomorphisms (projections) $\pi_1 , \pi_2 \in \ed_k(V)$ satisfying $\f=\pi_1 \circ \g = \g\circ \pi_2.$ 
\end{itemize}
\end{thm}

\section{GD1 and 1GD inverses of finite potent endomorphisms}\label{s: GD1 and 1GD}

The purpose of the present section is to generalize the theory of $GD1$ and $1GD$ inverses to arbitrary vector spaces, including infinite dimensional ones, using the theory of finite potent endomorphisms. Precisely, the aim is to study these inverses by means of the AST decomposition of finite potent endomorphisms (which is intrinsic to the operator as it is unique and characterizes it). Furthermore, we will define several binary relations in the set of finite potent endomorphisms (which generalize the ones presented in \cite{GD1N}) and characterize them as well.

\subsection{GD1 inverses of finite potent endomorphisms}\label{ss: GD1 FP}
Let $V$ denote an arbitrary vector space over an arbitrary ground field $k.$ 
\begin{defn}\label{D: GD1 }
Let $\f \in \ed_k(V)$ be a finite potent endomorphism. We will call ``generalized Drazin 1 inverse'' of $\f,$ or simply by $GD1$ inverse, to any linear operator, denoted as $\f^{GD1}$, satisfying that: \begin{align*}
\f^{GD1}\circ \f \circ \f^{GD1}& =\f^{GD1}\\
\f^{GD1}\circ \f & = \f^{GD}\circ \f \\
\f\circ \f^{GD1} & = \f \circ \f^-,
\end{align*} 
for certain $\f^-\in X_{\f}(1)$ and $\f^{GD}\in X_{\f}(GD).$
\end{defn}

\begin{thm}\label{T: Existencia}
Let $\f\in \ed_k(V)$ be a finite potent endomorphism. Then there exist GD1 inverses of $\f$; this is, for any $\f^-\in X_{\f}(1)$ and $\f^{GD}\in X_{\f}(GD),$ the system: \begin{align*}
\g \circ \circ \f \circ \g & = \g \\
\g \circ \f & = \f^{GD} \circ \f \\
\f \circ \g &= \f \circ \f^{-}
\end{align*} has solutions.
\end{thm}
\begin{proof}
It is clear that that $\g=\f^{GD}\circ \f \circ \f^{-}$ for any $\f^{-}\in X_{\f}(1)$ and $\f^{GD}\in X_{\f}(GD)$ is a solution for the system. If we now fix the $\f^-$ and $\f^{GD},$ then the solution is unique. Let us suppose that $\g$ is a solution. Then:
$$\g=\g\circ \f \circ \g=\f^{GD}\circ \f \circ \g=\f^{GD}\circ \f \circ \f^- $$ and we conclude.
\end{proof}

\begin{lem}\label{L: GD1 inv conj}
Let $\f \in \ed_k(V)$ be a finite potent endomorphism and let $\tau \in \mathrm{Aut}_k(V)$ be an automorphism of the $k-$vector space $V.$ Then, for any $\f^{GD1}\in X_{\f}(GD1),$ one has that $$\f^{GD1}\in X_{\f}(GD1) \text{ if and only if }\tau \circ \f^{GD1} \circ \tau^{-1}\in X_{\tau \circ \f \circ \tau^{-1}}(GD1). $$ 
\end{lem}
\begin{proof}
The reader can check this directly using Definition \ref{D: GD1 }.
\end{proof}

\begin{rem}
We shall highlight that for any finite potent endomorphism $\f\in \ed_k(V)$ the set $X_{\f}(1)$ is not empty and the set $X_{\f}(GD)\subseteq X_{\f}(1)$ is neither. Therefore, the set $X_{\f}(GD1)$ of $GD1$ inverses of a finite potent endomorphism is not empty.
\end{rem}

\begin{prop}\label{T: Charact Algebraica}
Let $\f \in \ed_k(V)$ be a finite potent endomorphism. Then, an endomorphism $\g\in \ed_k(V)$ satisfies that $\g=\f^{GD1}$ (in the sense of Definition \ref{D: GD1 }) if and only if $\g=\f^{GD}\circ \f\circ \f^{-}$ for certain $\f^-\in X_{\f}(1)$ and $\f^{GD}\in X_{\f}(GD).$ 
\end{prop}
\begin{proof}
Let us suppose that $\g$ satisfies the three conditions of Definition \ref{D: GD1 }. Then, $\g\circ \f \circ \f^-=\f^{GD}\circ \f \circ \f^-.$ Therefore, we conclude if we show that $\g\circ \f \circ \f^-=\g.$ But this is clear since $\g\circ \f \circ \f^-=\g\circ \f \circ \g=\g.$ Conversely, if $\g=\f^{GD}\circ \f \circ \f^-$ for certain $\f^-\in X_{\f}(1)$ and $\f^{GD}\in X_{\f}(GD),$ a straightforward calculation shows that the conditions of \ref{D: GD1 } are satisfied.
\end{proof}

\begin{lem}\label{L: GD1 son reflex}
Let $\f \in \ed_k(V)$ be a finite potent endomorphism. Then, every GD1 inverse is a reflexive generalized inverse.
\end{lem}
\begin{proof}
Let $\f^{GD1}\in X_{\f}(GD1).$ Clearly:
$\f \circ \f^{GD1}\circ \f = \f \circ \f^{-}\circ \f=\f, $ and we conclude by Definition \ref{D: GD1 }.
\end{proof}

\begin{cor}\label{C: GD contenidas GD1}
Given a finite potent endomorphism $\f\in \ed_k(V),$ for any $\f^{GD1}=\f^{GD}\circ \f \circ \f^{-}$ with $\f^-\in X_{\f}(1)$ and $\f^{GD}\in X_{\f}(GD),$ then:
\begin{itemize}
\item $\f^s\circ \f^{GD1}=\f^s\circ \f^-;$\smallskip
\item $\f^{GD1}\circ \f^s=\f^{GD}\circ \f^s$
\end{itemize}
for every positive integer $s.$ In particular, if $s=m=i(\f)$ then: $\f^{GD1}\circ \f^m=\f^{GD}\circ \f^m=\f^m\circ \f^{GD}.$
\end{cor}
\begin{proof}
It follows from the algebraic characterization obtained in Proposition \ref{T: Charact Algebraica}. Precisely: 
\begin{align*}
\f^s\circ \f^{GD1}&=\f^s\circ \f^{GD}\circ \f \circ \f^-=\f^{s-1}\circ (\f \circ \f^{GD}\circ \f)\circ \f^-=\f^{s}\circ \f^-;\\
\f^{GD1}\circ\f^s &=\f^{GD}\circ (\f \circ \f^-\circ \f)\circ \f^{s-1}=\f^{GD}\circ \f^s.
\end{align*}
The last statement is by definition of G-Drazin inverse.
\end{proof}

Our purpose now is to characterize the set of all $GD1$ inverses of a finite potent endomorphism $\f \in \ed_k(V)$ of index $m,$ the set $X_{\f}(GD1).$ With the notation of Section \ref{ss:basis-fp.endormo}, if $\mathrm{dim}_k(W_{\f})=r,$ let us fix a basis $B_V=B_{W_{\f}}\cup B_{U_{\f}}$ of $V$ with $$B_{W_{\f}}=\{w_1,\dots,w_r\} \text{ and }  B_{U_\varphi} =  \underset {\begin{aligned} s_i &\in S_{{\mu}_i (U_\varphi, \varphi)} \\ 1 &\leq i \leq m \end{aligned}} {\bigcup} \{{ {v_{s_i}}}, \varphi ({{ v_{s_i}}}), \dots , \varphi^{i-1} ({ {v_{s_i}}})\}.$$ In this conditions, recall that \begin{equation}\label{eq: GenNucleo}
\Ker( \varphi )=  \underset {\begin{aligned} s_i &\in S_{{\mu}_i (U_\varphi, \varphi)} \\ 1 &\leq i \leq m \end{aligned}} {\bigoplus} \langle \varphi^{i-1} ({ {v_{s_i}}}) \rangle.
\end{equation}

If $C=(c_{ij})$ is the matrix associated to $\f_{\vert_{W_{\f}}}$ in the basis $B_{W_{\f}},$ we have that $$\f(w_{j})=\sum_{i=1}^{r}c_{ij}w_i $$ for every $j\in \{1,\dots,r \}.$\\ Let us consider $\f^-, \f^{GD}\in \ed_k(V)$ some arbitrary $\{1\}-$inverse and G-Drazin inverse of $\f.$ On one hand, it follows from Proposition \ref{P: 1invfp} that:
\begin{align*}
\f^-(w_h)& =(\f_{\vert_{W_{\f}}})^{-1}(w_h)+u_h;\\
\f^-(\f^j(v_{s_i}))&=\f^{j-1}(v_{s_i})+u^j_{s_i};\\
\f^-(v_{s_i}) & =\tilde{v}_{s_i}
\end{align*}
for an arbitrary $\tilde{v}_{s_i}\in V$ and let us rewrite the last condition as 
\begin{equation}\label{eq: fmenosgen}
\f^-(v_{s_i})= \sum_{j=1}^r \gamma_j^{s_i}\cdot w_j + \underset {\begin{aligned} s_{i''} &\in S_{{\mu}_{i''} (U_\varphi, \varphi)} \\ 1 &\leq i'' \leq m \end{aligned}} \sum [\sum_{l=0}^{i''-1} \xi_{s_{i''}}^{s_i,l} \varphi^{l} ({{v_{s_{i''}}}})],
\end{equation}
where $\gamma_j^{s_i}\in k$ for all $j\in \{1,\dots,r\},$ $\xi_{s_{i''}}^{s_i,l}=0$ for almost all $s_{i''}\in S_{{\mu}_{i''}(U_\varphi, \varphi)},$ and $u_h,u^j_{s_i} \in \Ker(\f)$ for each $h\in \{1,\dots,r\}$ and for every $ s_i \in S_{{\mu}_i (U_{\f}, \f)}$  and $j\in \{1, \dots, i-1\}.$
On the other hand, by Proposition \ref{P: Charact GDrazinfp} we know that: 
\begin{align*}
\f^{GD}(w_h)&=(\f_{\vert_{ W_{\f}}})^{-1}(w_h);\bigskip \\
\f^{GD}(\f^j(v_{s_i}))&=\f^{j-1}(v_{s_i})+u^{'j}_{s_i};\\
\f^{GD}(v_{s_i})&=u_{s_i}\in U_{\f} \text{ for every }s_i\in S_{{\mu}_{i}(U_\varphi, \varphi)}
\end{align*}
 with  $u'^j_{s_i}\in \Ker(\f)$ for every  $j\in\{1,\dots,i-1 \}$ and $s_i\in S_{{\mu}_{i}(U_\varphi, \varphi)}.$ Bearing in mind \eqref{eq: GenNucleo}, let us rewrite: \begin{equation}\label{eq fgdfjota}
 \f^{GD}(\f^j(v_{s_i}))=\f^{j-1}(v_{s_i})+\underset {\begin{aligned} s_{i'} &\in S_{{\mu}_{i'} (U_\varphi, \varphi)} \\ 1 &\leq i' \leq m \end{aligned}} \sum \beta_{s_{i'}}^{s_i,j}\cdot \varphi^{i'-1} ({{v_{s_{i'}}}})
 \end{equation}
where $\beta_{s_{i'}}^{s_i,j}=0$ for almost all $s_{i'}\in S_{{\mu}_{i'} (U_\varphi, \varphi)}$ and $j\in \{1,\dots,i-1\}.$

\begin{prop}\label{P: CharactGD1FP}
Let $V$ be an arbitrary $k-$vector space and let us consider a finite potent endomorphism $\f \in \ed_k(V)$ of index $m.$ With the previous notations for an arbitrary $\f^-$ and $\f^{GD}$, an endomorphism $\f^{GD1}\in \ed_k(V)$ is a $GD1$ inverse of $\f$ if and only if it satisfies that
\begin{itemize}
\item $\f^{GD1}(w_h)=(\f_{\vert_{ W_{\f}}})^{-1}(w_h);$
\item $\f^{GD1}(\f^j(v_{s_i}))=\f^{j-1}(v_{s_i})+u'^j_{s_i};$ 
\item $\f^{GD1}(v_{s_i})= \sum_{j=1}^r \gamma_j^{s_i}\cdot w_j + \underset {\begin{aligned} s_{i'} &\in S_{{\mu}_{i'} (U_\varphi, \varphi)} \\ 1 &\leq i' \leq m \end{aligned}} \sum [\sum_{l=0}^{i'-1} \xi_{s_{i'}}^{s_i,l} \varphi^{l} ({{v_{s_{i'}}}})]$
with
$$\xi_{s_{i'}}^{s_i,i'-1} = \underset {\begin{aligned} s_{i''} &\in S_{{\mu}_{i''} (U_\varphi, \varphi)} \\ 1 &\leq i'' \leq m \end{aligned}} \sum {( \sum_{l=0}^{i''-2} [\xi_{s_{i''}}^{s_i,l} \cdot \beta_{s_{i'}}^{s_{i''},l+1}] )}\, 
$$\noindent 
\end{itemize}with  $u'^j_{s_i}\in \Ker(\f)$ for every  $j\in\{1,\dots,i-1 \}$ and $s_i\in S_{{\mu}_{i}(U_\varphi, \varphi)},$ and with  $\xi_{s_{i'}}^{s_i,l} = 0$ for almost all  $s_{i'} \in S_{{\mu}_i (U_\varphi, \varphi)}$ and $l\in \{1, \dots, i-1\}$.

\end{prop}
\begin{proof}
Firstly, if there is an operator $\f^{GD1}$ with the above expression, it is immediate from Proposition \ref{P: Char Reflex Fp} that $\f^{GD1}\in X_{\f}(1,2).$ Moreover, it is easy to define a $\{1\}-$inverse $\f^-$ and a G-Drazin inverse $\f^{GD}$ such that $\f^{GD1}\circ \f=\f^{GD}\circ \f$ and $\f \circ \f^{GD1}=\f\circ \f^-$ and so, Definition \ref{D: GD1 } is satisfied.

Conversely, if $\f^{GD1}\in \ed_k(V)$ is a $GD1$ inverse, let us suppose that $\f^{GD1}=\f^{GD}\circ \f \circ \f^{-}$ (Proposition \ref{T: Charact Algebraica}) for some arbitrary $\f^{-}$ and $\f^{GD}.$ Without any loss of generality, let us suppose that $\f^-$ and $\f^{GD}$ have the expressions mentioned above. From the first two conditions in the statement it is straightforward to see that $(\f^{GD1})_{\vert_{\Img(\f)}}=(\f^{GD})_{\vert_{\Img(\f)}}.$ Hence, we shall check the last of the three conditions, precisely that $\f^{GD1}(v_{s_i})=(\f^{GD}\circ \f \circ \f^-)(v_{s_i})$ for every $s_i\in S_{{\mu}_i (U_\varphi, \varphi)}$ and for all $i\in \{1,\dots,m\}$ in order to prove the claim. Recalling from \eqref{eq fgdfjota} that $$ \f^{GD}(\f^j(v_{s_i}))=\f^{j-1}(v_{s_i})+\underset {\begin{aligned} s_{i'} &\in S_{{\mu}_{i'} (U_\varphi, \varphi)} \\ 1 &\leq i' \leq m \end{aligned}} \sum \beta_{s_{i'}}^{s_i,j}\cdot \varphi^{i'-1} ({{v_{s_{i'}}}}), $$\noindent
 and bearing in mind \eqref{eq: fmenosgen}, one has that 

$$\begin{aligned} &(\f^{GD}\circ \f \circ \f^-)(v_{s_i}) =  (\f^{GD}\circ \f) \big (\sum_{j=1}^r \gamma_j^{s_i}\cdot w_j + \underset {\begin{aligned} s_{i''} &\in S_{{\mu}_{i''} (U_\varphi, \varphi)} \\ 1 &\leq i'' \leq m \end{aligned}} \sum [\sum_{l=0}^{i''-1} \xi_{s_{i''}}^{s_i,l} \varphi^{l} ({{v_{s_{i''}}}})] \big )  \\ &=  \f^{GD} \big (\sum_{j=1}^r \gamma_j^{s_i}\cdot \varphi (w_j)+ \underset {\begin{aligned} s_{i''} &\in S_{{\mu}_{i''} (U_\varphi, \varphi)} \\ 1 &\leq i'' \leq m \end{aligned}} \sum [\sum_{l=0}^{i''-1} \xi_{s_{i''}}^{s_i,l} \varphi^{l+1} ({{v_{s_{i''}}}})]  \big )  \\ &=  \f^{GD} \big (\sum_{j,h=1}^r [\gamma_j^{s_i} \cdot c_{hj}\cdot w_h]+ \underset {\begin{aligned} s_{i''} &\in S_{{\mu}_{i''} (U_\varphi, \varphi)} \\ 1 &\leq i'' \leq m \end{aligned}} \sum [\sum_{l=0}^{i''-1} \xi_{s_{i''}}^{s_i,l} \varphi^{l+1} ({{v_{s_{i''}}}})]  \big )\\=&\sum_{j=1}^r \gamma_j^{s_i}\cdot [\sum_{h=1}^r c_{hj}\cdot (\varphi_{\vert_{W_\varphi}})^{-1} (w_h) ]+ \\ &+ \underset {\begin{aligned} s_{i''} &\in S_{{\mu}_{i''} (U_\varphi, \varphi)} \\ 1 &\leq i'' \leq m \end{aligned}} \sum \big ( [\sum_{l=0}^{i''-2} \xi_{s_{i''}}^{s_i,l} \varphi^{l} ({{v_{s_{i''}}}})] + [\sum_{l=0}^{i''-2} \xi_{s_{i''}}^{s_i,l} \cdot \big [\underset {\begin{aligned} s_{i'} &\in S_{{\mu}_{i'} (U_\varphi, \varphi)} \\ 1 &\leq i' \leq m \end{aligned}} \sum \beta_{s_{i'}}^{s_{i''},l+1}\cdot \varphi^{i'-1} ({ {v_{s_{i'}}}})\big ]  \big ) \end{aligned}$$ which is equal to $$\begin{aligned} &= \sum_{j=1}^r \gamma_j^{s_i}\cdot w_j +\\ &+ \underset {\begin{aligned} s_{i'} &\in S_{{\mu}_{i' } (U_\varphi, \varphi)} \\ 1 &\leq i' \leq m \end{aligned}} \sum \big ( [\sum_{l=0}^{i'-2} \xi_{s_{i'}}^{s_i,l} \varphi^{l} ({{v_{s_{i'}}}})] +  \underset {\begin{aligned} s_{i''} &\in S_{{\mu}_{i''} (U_\varphi, \varphi)} \\ 1 &\leq i'' \leq m \end{aligned}} \sum {( \sum_{l=0}^{i''-2} [\xi_{s_{i''}}^{s_i,l} \cdot \beta_{s_{i'}}^{s_{i''},l+1}] )}\big )\cdot \varphi^{i'-1} ({ {v_{s_{i'}}}})\, ,\end{aligned}$$

\noindent from where the assertion is deduced.\end{proof}

\begin{rem}
It shall be pointed out that, in general, not every $GD1$ inverse of a finite potent endomorphism is a finite potent operator.
\end{rem}

\begin{thm}\label{T: Caract GD1 INV AST}
Let us consider a finite potent endomorphism $\f \in \ed_k(V)$ that induces an AST decomposition $V=W_{\f}\oplus U_{\f}.$ Then, an endomorphism $\g\in \ed_k(V)$ is a GD1 inverse of $\f$ if and only if $\g\in X_{\f}(1,2)$ and $\g(W_{\f})\subseteq W_{\f}$ (it leaves $W_{\f}$ invariant).
\end{thm}
\begin{proof}
Let us suppose that $\g=\f^{GD1}.$ Then, $(\f^{GD1})_{\vert_ {W_{\f}}}=(\f_{\vert_{W_{\f}}})^{-1}$ and hence it leaves $W_{\f}$ invariant. We already proved in Lemma \ref{L: GD1 son reflex} that $\f^{GD1}$ is a reflexive generalized inverse. Conversely, let us suppose that there is an endomorphism that satisfies the conditions in the statement. Moreover, let us suppose that it has the expression described in Proposition \ref{P: Char Reflex Fp}. In particular, one has that $\g(w_h)=(\g_{\vert_{W_\varphi}})^{-1} (w_h) + u_h;$ with $u_h \in \Ker(\f)$ for every $h\in \{ 1,\dots ,r\}.$ We can write $u_h=\underset {\begin{aligned} s_{i'} &\in S_{{\mu}_{i'} (U_\varphi, \varphi)} \\ 1 &\leq i' \leq n \end{aligned}} \sum \lambda_{s_{i'}}^h\cdot \varphi^{i'-1} ({ {v_{s_{i'}}}})\,$ ,\noindent with $\lambda_{s_{i'}}^h\in k$ for each $s_{i'} \in S_{{\mu}_{i'} (U_\varphi, \varphi)}$ and each $h\in \{1, \dots, r\}$ and where only a finite number of the scalars $\{\lambda_{s_{i'}}^h\}$ are different from zero. As $\g(W_{\f})\subseteq W_{\f},$ we deduce that $u_h=0$ for every $h\in \{1,\dots,r \}.$ In particular $\lambda_{s_{i'}}^h=0$ for each $s_{i'} \in S_{{\mu}_{i'} (U_\varphi, \varphi)}$ and each $h\in \{1, \dots, r\}.$ Hence, readers can easily check that the expression of $\g$ in the basis is precisely the one described in Proposition \ref{P: CharactGD1FP} and we conclude.
\end{proof}

\begin{thm}[\textbf{Structure of $X_{\f}(GD1)$}]\label{T: EstructuraGD1}
If $\f\in \ed_k(V)$ is a finite potent endomorphism of index $m,$ then there exists a bijection: \begin{equation} X_\varphi (GD1) \simeq \big [\underset {\begin{aligned} s_i &\in S_{{\mu}_i (U_\varphi, \varphi)} \\ 1 &\leq i \leq m \end{aligned}} {\prod} [(V/\Ker \varphi) \times \prod_{j= 1}^{i-1} \Ker \varphi] \big ]\, .\end{equation}
\end{thm}
\begin{proof}
With the notation used in Proposition \ref{P: CharactGD1FP}, if we denote:

\begin{itemize}

\item $u_{s_i}^{'j} = \underset {\begin{aligned} s_{i'} &\in S_{{\mu}_{i'} (U_\varphi, \varphi)} \\ 1 &\leq i' \leq m \end{aligned}} \sum \beta_{s_{i'}}^{s_i,j}\cdot \varphi^{i'-1} ({ {v_{s_{i'}}}}) \in \Ker \varphi$;

\bigskip

\item $[{\tilde v}_{s_i}] = \big [\sum_{j=1}^r \gamma_j^{s_i}\cdot w_j +  \underset {\begin{aligned} s_{i'} &\in S_{{\mu}_{i' } (U_\varphi, \varphi)} \\ 1 &\leq i' \leq m \end{aligned}} \sum \big ( [\sum_{l=0}^{i'-2} \xi_{s_{i'}}^{s_i,l} \varphi^{l} ({{v_{s_{i'}}}})] \big ] \in V/\Ker \varphi$,
\end{itemize}\noindent

it follows from Proposition \ref{P: CharactGD1FP} that the map $$\begin{aligned} X_\varphi (GD1) &\longrightarrow \big [\underset {\begin{aligned} s_i &\in S_{{\mu}_i (U_\varphi, \varphi)} \\ 1 &\leq i \leq m \end{aligned}} {\prod} [(V/\Ker \varphi) \times \prod_{j= 1}^{i-1} \Ker \varphi] \big ] \\ \f^{GD1} &\longmapsto (([{\tilde v}_{s_i}], u_{s_i}^j)_{ j\in \{1, \dots, i-1\}})_{s_i \in S_{{\mu}_i (U_\varphi, \varphi)}}\end{aligned}$$\noindent is a bijection.
\end{proof}

\subsubsection{Computation of the GD1-inverses of a finite square matrix}\label{ss: Computation GD1 Matrices}

In the present section we apply the previous results to characterize the set $A(GD1)$ of GD1 inverses of a finite square matrix $A$ with entries in an arbitrary field $k$ and we offer an algorithm for the explicit computation of $A(GD1).$ 

\begin{thm} [Structure of $A(GD1)$] \label{T: Char AGD1 fin} Let $A \in \text{Mat}_{n\times n} (k)$ be a square matrix with entries in an arbitrary field $k$ and let $A = A_1 + A_2$ be its core-nilpotent decomposition. Then, the structure of $A(GD1)$ is determined from the following bijection:
\begin{equation} \label{eq:char AGD1} A(GD1) \simeq  [\prod_{i=1}^{\text{rk} \, A_2} N_u(A)] \times [\prod_{i=1}^{\text{dim } N_u(A)} k^{\text{rk} \, A}]\simeq k^{[\mathrm{dim } N_u(A)]\cdot (\mathrm{rk (A)}+\mathrm{rk (A_2)}) }   \, . \end{equation}
\end{thm}

\begin{proof} Bearing in mind the well-known relationship between finite square matrices and endomorphisms of finite-dimensional vector spaces, the statement is  immediately deduced from Theorem \ref{T: EstructuraGD1}.
\end{proof}

Accordingly, an algorithm for computing the set $A(GD1)$ for certain $A\in \Mat_{n\times n}(k)$ is the following:

\begin{enumerate}

\item Write $A$ in its Jordan canonical form: $A = P\cdot J \cdot P^{-1}$. \medskip

\item If $J = \begin{pmatrix} C & 0 \\ 0 & N \end{pmatrix}$ with $C\in \text{Mat}_{r\times r} (k)$ invertible and $N$ nilpotent, let $\{m_1, m_2, \dots, m_s\}$ be the set of natural numbers such that $m_i \leq m_{i+1}$ and each $(r+m_i)$-column of $J$ are zero.  \smallskip

\item Compute the inverse $C^{-1}$ and the transpose $N^t.$

\item Calculate the nullspace $N_u(J)$.\smallskip

\item Put $J' = \begin{pmatrix} C^{-1} & 0 \\ 0 & N^t \end{pmatrix}$.

\item\label{eq: item6enumer} Add a general element of $N(J)$ in the non-zero columns of $J'$ except for the first $r$ ones to get a matrix $J''$. 

\item Obtain a matrix $\tilde {J}$ by completing the zero columns of $J''$ with arbitrary parameters except the $i$- rows of these columns with $$i\in \{ r+ m_1, r+ m_2, \dots, r+ m_{s}\}\, .$$\noindent Note that the zero columns of $J''$ are the $j$-columns with $$j\in \{ r+ 1, r+ m_1 +1, \dots, r+ m_{s-1} +1\}\, .$$

\item Setting $C = (c_{ij})$ and ${\tilde J} = (\alpha'_{ij})$, we write $J^{GD1} = (\alpha_{ij})$ with $$\alpha_{ij} = \begin{pmatrix} \alpha'_{i1} & \alpha'_{i2} & \dots & \alpha'_{in} \end{pmatrix} \cdot J \cdot \begin{pmatrix} \alpha'_{1j} \\ \alpha'_{2j} \\ \vdots \\ \alpha'_{nj} \end{pmatrix}  \, ,$$\noindent for each $i\in \{ r+ m_1, \dots, r+ m_{s}\}$ and $j\in \{ r+ 1, r+ m_1 +1, \dots, r+ m_{s-1} +1\}$ and being $\alpha_{hs} = \alpha'_{hs}$ otherwise. \medskip

\item Compute any $A^{GD1} \in A(GD1)$ depending on $[\mathrm{dim } N_u(A)]\cdot (\mathrm{rk (A)}+\mathrm{rk (A_2)}) $ parameters as $A^{GD1} = P\cdot J^{GD1}\cdot P^{-1}$.
\end{enumerate} 

\smallskip

\subsection{1GD inverses of finite potent endomorphisms}\label{ss: 1GD fp}

The following section is devoted to the study of 1GD inverses, which can be understood as the dual generalized inverses of the GD1 inverses. Due to the analogy with their respective duals, we will only list the results concerning them without proofs unless we deem it necessary or clarifying.

\begin{defn}\label{D: 1GD }
Let $\f\in \ed_k(V)$ be a finite potent endomorphism. We will call ``1 generalized Drazin inverse'' of $\f,$ or simply, 1GD inverse  to any linear operator,  denoted as $\f^{1GD},$ satisfying that: \begin{align*}
\f^{1GD}\circ \f \circ \f^{1GD} & = \f^{1GD}\\
\f^{1GD}\circ \f & = \f^{-}\circ \f \\
\f \circ \f^{1GD} &= \f \circ \f^{GD}. 
\end{align*}
\end{defn}

\begin{thm}\label{T: Existencia 1GD}
Let us consider a finite potent endomorphism $\f \in \ed_k(V).$ There exists 1GD inverses of $\f,$ this is, the system: \begin{align*}
\g\circ \f \circ \g &= \g \\
\g \circ \f & = \f^- \circ \f \\
\f \circ \g & = \f \circ \f^{GD}
\end{align*} has solutions for every $\f^-\in X_{\f}(1)$ and every $\f^{GD}\in X_{\f}(GD).$
\end{thm}

\begin{lem}\label{L: 1GD inv conj}
Let us consider a finite potent endomorphism $\f \in \ed_k(V)$ and let $\tau \in \mathrm{Aut}_k(V)$ be an automorphism of the $k-$vector space $V.$ Then, for any $\f^{1GD}\in X_{\f}(1GD),$ one has that $$\f^{1GD}\in X_{\f}(1GD) \text{ if and only if }\tau \circ \f^{1GD} \circ \tau^{-1}\in X_{\tau \circ \f \circ \tau^{-1}}(1GD). $$ 
\end{lem}

\begin{lem}\label{L: 1GD son Reflex}
Let $\f \in \ed_k(V)$ be a finite potent endomorphism. Then, every 1GD inverse of $\f$ is a reflexive generalized inverse.
\end{lem}

\begin{prop}\label{P: Charact Algebraica 1GD}
Let us consider a finite potent endomorphism $\f\in \ed_k(V).$ Then, a linear operator $\g\in \ed_k(V)$ satisfies that $\g =\f^{1GD}$ in the sense of Definition \ref{D: 1GD } if and only if $\g=\f^-\circ \f \circ \f^{GD}$ for certain $\f^{-}\in X_{\f}(1)$ and $\f^{GD}\in X_{\f}(GD).$
\end{prop}

\begin{cor}\label{C: 1GD conm ind}
Given a finite potent endomorphism $\f \in \ed_k(V)$ for any $\f^{1GD}=\f^-\circ \f \circ \f^{GD}$ with $\f^{-}\in X_{\f}(1)$ and $\f^{GD}\in X_{\f}(GD),$ then: \begin{itemize}
\item $\f^s\circ \f^{1GD}=\f^s \circ \f^{GD},$
\item $\f^{1GD}\circ \f^s=\f^-\circ \f^s,$
\end{itemize} for any positive integer $s.$ In particular, if $s=m=i(\f),$ then $\f^m\circ \f^{1GD}=\f^m\circ \f^{GD}=\f^{GD}\circ \f^m.$
\end{cor}

\begin{cor}
Let $\f \in \ed_k(V)$ be a finite potent endomorphism and let us consider $\f^{GD1}=\f^{GD}\circ \f \circ \f^-$ and $\f^{1GD}=\f^-\circ \f \circ \f^{GD}$ with $\f^{-}\in X_{\f}(1)$ and $\f^{GD}\in X_{\f}(GD).$ If $m=i(\f),$ then: $$\f^m\circ \f^{1GD}=\f^{GD1}\circ \f^m. $$
\end{cor}
\begin{proof}
It is a direct consequence of Corollary \ref{C: GD contenidas GD1} and Corollary \ref{C: 1GD conm ind}.
\end{proof}

Let us continue by giving the characterization of the set of all 1GD inverses of a finite potent endomorphism $\f \in \ed_k(V)$ of index $m,$ this is, the set $X_{\f}(1GD).$ Let us maintain the notation as similar as possible to the previous section.

 Let us consider $\f^-, \f^{GD}\in \ed_k(V)$ an arbitrary $\{1\}-$inverse and G-Drazin inverse of $\f.$ From Proposition \ref{P: 1invfp} we know that:
\begin{align*}
\f^-(w_h)& =(\f_{\vert_{W_{\f}}})^{-1}(w_h)+u_h;\\
\f^-(\f^j(v_{s_i}))&=\f^{j-1}(v_{s_i})+ \underset {\begin{aligned} s_{i''} &\in S_{{\mu}_{i''} (U_\varphi, \varphi)} \\ 1 &\leq i'' \leq m \end{aligned}} \sum [\sum_{l=0}^{i''-1} \alpha_{s_{i''}}^{s_i,l} \varphi^{l} ({{v_{s_{i''}}}})]\\
\f^-(v_{s_i}) & =\tilde{v}_{s_i}
\end{align*}

where $\tilde{v}_{s_i}\in V$ for every $s_{i}\in S_{{\mu}_{i}(U_\varphi, \varphi)},$  $\alpha_{s_{i''}}^{s_i,l}=0$ for almost all $s_{i''}\in S_{{\mu}_{i''}(U_\varphi, \varphi)}$ and $j\in \{1,\dots,i-1\},$ and $u_h \in \Ker(\f)$ for each $h\in \{1,\dots,r\}.$ In fact, by Proposition \ref{P: Charact GDrazinfp}, we can ensure that: 
\begin{align*}
\f^{GD}(w_h)&=(\f_{\vert_{ W_{\f}}})^{-1}(w_h);\bigskip \\
\f^{GD}(\f^j(v_{s_i}))&=\f^{j-1}(v_{s_i})+u_{s_i}^{'j};\\
\f^{GD}(v_{s_i})&=\underset {\begin{aligned} s_{i'} &\in S_{{\mu}_{i'} (U_\varphi, \varphi)} \\ 1 &\leq i' \leq m \end{aligned}} \sum [\sum_{l=0}^{i'-1} \epsilon_{s_{i'}}^{s_i,l} \varphi^{l} ({{v_{s_{i'}}}})] \text{ for every }s_i\in S_{{\mu}_{i}(U_\varphi, \varphi)}
\end{align*}
 with  $u_{s_i}^{'j}\in \Ker(\f)$ for every $s_i\in S_{{\mu}_{i}(U_\varphi, \varphi)}$ and $j\in \{1,\dots,i-1\},$ and with $\epsilon_{s_{i'}}^{s_i,l}=0$ for almost all $s_{i'}\in S_{{\mu}_{i'} (U_\varphi, \varphi)}.$ 

\begin{prop}\label{P: Charact1GDFP}
Let $V$ be an arbitrary $k-$vector space and let us consider a finite potent endomorphism $\f \in \ed_k(V)$ of index $m.$ With the previous notations for an arbitrary $\f^-$ and $\f^{GD}$, an endomorphism $\f^{1GD}\in \ed_k(V)$ is a $1GD$ inverse of $\f$ if and only if it satisfies that \begin{itemize}
\item $\f^{1GD}(w_h)=(\f_{\vert_{W_{\f}}})^{-1}(w_h)+u_h;$
\item $\f^{1GD}(\f^j(v_{s_i}))=\f^{j-1}(v_{s_i})+ \underset {\begin{aligned} s_{i''} &\in S_{{\mu}_{i'} (U_\varphi, \varphi)} \\ 1 &\leq i' \leq m \end{aligned}} \sum [\sum_{l=0}^{i'-1} \alpha_{s_{i'}}^{s_i,l} \varphi^{l} ({{v_{s_{i'}}}})]; \medskip $
\item $\f^{1GD}(v_{s_i})=\underset {\begin{aligned} s_{i'} &\in S_{{\mu}_{i'} (U_\varphi, \varphi)} \\ 1 &\leq i' \leq m \end{aligned}} \sum [\sum_{l=0}^{i'-1} \epsilon_{s_{i'}}^{s_i,l} \varphi^{l} ({{v_{s_{i'}}}})]$ with \begin{equation}
\epsilon_{s_{i'}}^{s_i,i'-1} = \underset {\begin{aligned} s_{i''} &\in S_{{\mu}_{i''} (U_\varphi, \varphi)} \\ 1 &\leq i'' \leq m \end{aligned}} \sum {( \sum_{l=0}^{i''-2} [\epsilon_{s_{i''}}^{s_i,l} \cdot \alpha_{s_{i'}}^{s_{i''},l+1}] )}\, 
,\end{equation}
 were $u_h\in \Ker(\f)$ for every $h\in \{1,\dots,r \},$ $\alpha_{s_{i'}}^{s_i,l}=0$ and $\epsilon_{s_{i'}}^{s_i,l} = 0$ for almost all  $s_{i'} \in S_{{\mu}_i (U_\varphi, \varphi)}$ and $l\in \{1, \dots, i-1\}$.
\end{itemize}

\end{prop}
\begin{proof}
The first part of the proof is analogous to that of Proposition \ref{P: CharactGD1FP}.

Conversely, if $\f^{1GD}\in \ed_k(V)$ is a $1GD$ inverse, let us suppose that $\f^{1GD}=\f^{-}\circ \f \circ \f^{GD}$ (Proposition \ref{P: Charact Algebraica 1GD}) for some arbitrary $\f^{-}$ and $\f^{GD}.$ Without any loss of generality, let us suppose that $\f^-$ and $\f^{GD}$ have the expressions mentioned above. Using the first two conditions in the statement it is direct to check that $(\f^{1GD})_{\vert_{\Img(\f)}}=(\f^{-})_{\vert_{\Img(\f)}}.$ Hence, we shall see that $\f^{1GD}(v_{s_i})=(\f^{-}\circ \f \circ \f^{GD})(v_{s_i})$ for every $s_i\in S_{{\mu}_i (U_\varphi, \varphi)}$ and for all $i\in \{1,\dots,m\}$ in order to prove the claim. Recalling the notation used above, one has that $$\begin{aligned} &(\f^{-}\circ \f \circ \f^{GD})(v_{s_i}) =  (\f^{-}\circ \f) \big (\underset {\begin{aligned} s_{i'} &\in S_{{\mu}_{i'} (U_\varphi, \varphi)} \\ 1 &\leq i' \leq m \end{aligned}} \sum [\sum_{l=0}^{i'-1} \epsilon_{s_{i'}}^{s_i,l} \varphi^{l} ({{v_{s_{i'}}}})])  \\ &=  \f^{-} \big (\underset {\begin{aligned} s_{i'} &\in S_{{\mu}_{i'} (U_\varphi, \varphi)} \\ 1 &\leq i' \leq m \end{aligned}} \sum [\sum_{l=0}^{i'-1} \epsilon_{s_{i'}}^{s_i,l+1} \varphi^{l} ({{v_{s_{i'}}}})])= \\ &= \underset {\begin{aligned} s_{i'} &\in S_{{\mu}_{i'} (U_\varphi, \varphi)} \\ 1 &\leq i' \leq m \end{aligned}} \sum [\sum_{l=0}^{i'-2} \epsilon_{s_{i'}}^{s_i,l} \varphi^{l} ({{v_{s_{i'}}}})]) + [\sum_{l=0}^{i''-2} \epsilon_{s_{i'}}^{s_i,l} \cdot \big [\underset {\begin{aligned} s_{i''} &\in S_{{\mu}_{i''} (U_\varphi, \varphi)} \\ 1 &\leq i'' \leq m \end{aligned}} \sum \alpha_{s_{i'}}^{s_{i''},l+1}\cdot \varphi^{i''-1} ({ {v_{s_{i''}}}})\big ]  \big )  \smallskip\\ &= \underset {\begin{aligned} s_{i'} &\in S_{{\mu}_{i'} (U_\varphi, \varphi)} \\ 1 &\leq i' \leq m \end{aligned}} \sum [\sum_{l=0}^{i'-2} \epsilon_{s_{i'}}^{s_i,l} \varphi^{l} ({{v_{s_{i'}}}})]) +  \underset {\begin{aligned} s_{i''} &\in S_{{\mu}_{i''} (U_\varphi, \varphi)} \\ 1 &\leq i'' \leq m \end{aligned}} \sum {( \sum_{l=0}^{i''-2} [\epsilon_{s_{i'}}^{s_i,l} \cdot \alpha_{s_{i'}}^{s_{i''},l+1}] )}\big )\cdot \varphi^{i'-1} ({ {v_{s_{i'}}}})\, ,\end{aligned}$$\noindent from where the statement is deduced.

\end{proof}

\begin{rem}
It shall be pointed out that, in general, not every $1GD$ inverse of a finite potent endomorphism is a finite potent operator.
\end{rem}

\begin{thm}\label{T: Charact 1GD AST e inva}
An endomorphism $\phi \in \ed_k(V)$ is a $1GD$ inverse of a finite potent endomorphism $\f$ that induces an AST decomposition $V=W_{\f}\oplus U_{\f}$ if and only if $\phi \in X_{\f}(1,2)$ and $\phi(U_{\f})\subseteq U_{\f}.$
\end{thm}

\begin{thm}[\textbf{Structure of $X_{\f}(1GD)$}]\label{T: Estructura1GD}
If $\f\in \ed_k(V)$ is a finite potent endomorphism of index $m,$ such that it induces an AST decomposition $V=W_{\f}\oplus U_{\f},$ then there exists a bijection: \begin{equation} X_\varphi (1GD) \simeq \prod_{j= 1}^{r} \Ker \varphi \times \big [\underset {\begin{aligned} s_i &\in S_{{\mu}_i (U_\varphi, \varphi)} \\ 1 &\leq i \leq m \end{aligned}} {\prod} [(U_{\f}/\Ker \varphi) \times \prod_{j= 1}^{i-1} \Ker \varphi] \big ]\, .\end{equation}
\end{thm}

\subsubsection{Computation of the 1GD inverses of a finite square matrix}\label{ss: Computation 1GD Matrices}

The aim of this section is to characterize the set $A(1GD)$ of 1GD inverses of a finite square matrix $A$ with entries in an arbitrary field k and moreover, to offer an algorithm to explicitly calculate $A(1GD).$

\begin{thm} [Structure of $A(1GD)$] \label{T: Char A1GD fin} Let $A \in \text{Mat}_{n\times n} (k)$ be a square matrix with entries in an arbitrary field $k$ and let $A = A_1 + A_2$ be its core-nilpotent decomposition. Then, the structure of $A(1GD)$ is determined from the following bijection:
\begin{equation} \label{eq:char A1GD} A(1GD) \simeq [\prod_{i=1}^{\mathrm{rk A_1}} N_u(A)]\times  [\prod_{i=1}^{\text{rk} \, A_2} N_u(A)] \times [\prod_{i=1}^{\text{dim } N_u(A)} k^{\text{rk} \, A}]\simeq k^{[\mathrm{dim }N_u(A)]\cdot (\mathrm{rk (A)}+\mathrm{rk (A_2)}) }  \, . \end{equation}
\end{thm}

\begin{proof} Bearing in mind the well-known relationship between finite square matrices and endomorphisms of finite-dimensional vector spaces, the statement is immediately deduced from Theorem \ref{T: Estructura1GD}, as \begin{align*}
& \mathrm{dim }N_{u}(A)\cdot \mathrm{rk (A_1)}+\mathrm{dim }N_u(A)\cdot( \mathrm{rk (A_2)}-\mathrm{dim }N_u(A))+\mathrm{dim }N_u(A)\cdot \mathrm{rk (A_2)}=\\ & \mathrm{dim }N_u(A)\cdot ((\mathrm{rk (A_1)}+\mathrm{rk (A_2)} -\mathrm{dim }N_u(A)+\mathrm{rk (A_2)})=\\ & \mathrm{dim }N_u(A)\cdot ((\mathrm{rk (A_1)}+\mathrm{rk (A_2)} -(m-\mathrm{rk (A)})+\mathrm{rk (A_2)})=\\ & \mathrm{dim }N_u(A)\cdot (\mathrm{rk (A)}+\mathrm{rk (A_2)}).
\end{align*} 
\end{proof}

Accordingly, an algorithm for computing the set $A(1GD)$ for certain $A\in \Mat_{n\times n}(k)$ is the following:

\begin{enumerate}

\item Write $A$ in its Jordan canonical form: $A = P\cdot J \cdot P^{-1}$. \medskip

\item If $J = \begin{pmatrix} C & 0 \\ 0 & N \end{pmatrix}$ with $C\in \text{Mat}_{r\times r} (k)$ invertible and $N$ nilpotent, let $\{m_1, m_2, \dots, m_s\}$ be the set of natural numbers such that $m_i \leq m_{i+1}$ and each $(r+m_i)$-column of $J$ are zero.  \smallskip

\item Compute the inverse $C^{-1}$ and the transpose $N^t.$

\item Calculate the nullspace $N_u(J)$.\smallskip

\item Put $J' = \begin{pmatrix} C^{-1} & 0 \\ 0 & N^t \end{pmatrix}$.

\item\label{eq: item6enumer} Add a general element of $N(J)$ in the non-zero columns of $J'.$ 

\item Obtain a matrix $\tilde {J}$ by completing the zero columns of $J''$ with arbitrary parameters except the first $r$ rows of these columns. \noindent Note that the zero columns of $J''$ are the $j$-columns with $$j\in \{ r+ 1, r+ m_1 +1, \dots, r+ m_{s-1} +1\}\, .$$

\item Setting $C = (c_{ij})$ and ${\tilde J} = (\alpha'_{ij})$, we write $J^{1GD} = (\alpha_{ij})$ with $$\alpha_{ij} = \begin{pmatrix} \alpha'_{i1} & \alpha'_{i2} & \dots & \alpha'_{in} \end{pmatrix} \cdot J \cdot \begin{pmatrix} \alpha'_{1j} \\ \alpha'_{2j} \\ \vdots \\ \alpha'_{nj} \end{pmatrix}  \, ,$$\noindent for each $i\in \{ r+ m_1, \dots, r+ m_{s}\}$ and $j\in \{ r+ 1, r+ m_1 +1, \dots, r+ m_{s-1} +1\}$ and being $\alpha_{hs} = \alpha'_{hs}$ otherwise. \medskip

\item Compute any $A^{1GD} \in A(1GD)$ depending on $[\mathrm{dim }N_u(A)]\cdot (\mathrm{rk (A)}+\mathrm{rk (A_2)})$ parameters as $A^{1GD} = P\cdot J^{1GD}\cdot P^{-1}$.
\end{enumerate} 

\smallskip

\subsection{Illustrative examples.}\label{ss. Examples Algorithms}
 
 To finish this part of the paper we shall offer an example of the explicit application of the algorithms presented in Sections \ref{ss: Computation GD1 Matrices} and \ref{ss: Computation 1GD Matrices} for the computation of the sets of GD1 and 1GD inverses of a square matrix. Precisely we will calculate these sets for the same matrix so that the differences are noted using the notation described in the aforementioned sections.

Let us consider the following matrix $$A = \begin{pmatrix}
19-4i & -12+3i & -9+2i & 20-4i & 15-4i \\
12-4i & -6+3i & -6+2i & 12-4i & 12-4i \\
186-12i & -117+9i & -87+6i & 192-12i & 144-12i \\
45-4i & -28+3i & -21+2i & 46-4i & 35-4i \\
38 & -23 & -18 & 39 & 31
\end{pmatrix} \in \text{Mat}_{5\times 5} ({\mathbb C})\, .$$

Moreover let us consider the following Jordan form $$A = \begin{pmatrix}
1 & 0 & 3 & -1 & 1 \\
1 & -2 & 0 & 0 & 0 \\
3 & 1 & 2 & 0 & 6 \\
1 & 0 & -2 & 1 & 1 \\
0 & -1 & 0 & 0 & 1
\end{pmatrix}\cdot \begin{pmatrix}
i & 0 & 0 & 0 & 0 \\
0 & 3 & 0 & 0 & 0 \\
0 & 0 & 0 & 0 & 0 \\
0 & 0 & 1 & 0 & 0 \\
0 & 0 & 0 & 1 & 0
\end{pmatrix} \cdot \begin{pmatrix}
-4 & 3 & 2 & -4 & -4 \\
-2 & 1 & 1 & -2 & -2 \\
13 & -8 & -6 & 13 & 10 \\
32 & -20 & -15 & 33 & 25 \\
-2 & 1 & 1 & -2 & -1
\end{pmatrix}\, .$$\noindent
\begin{exam}\label{Ex: Comp GD1}

 Let us calculate the set of GD1 inverses of matrix A making clear the notation in every step of the algorithm. Hence,
 $$J=\begin{pmatrix}
i & 0 & 0 & 0 & 0 \\
0 & 3 & 0 & 0 & 0 \\
0 & 0 & 0 & 0 & 0 \\
0 & 0 & 1 & 0 & 0 \\
0 & 0 & 0 & 1 & 0
\end{pmatrix}. $$\noindent In this case, we have that $C=\begin{pmatrix}
i & 0 \\
0 & 3
\end{pmatrix}\in Mat_{2 \times 2}(\mathbb{C})$ and $\{m_1\}$ with $m_{1}=3$, such that the $(2+3)=(5)-$column of J is zero. Clearly $C^{-1}=\begin{pmatrix}
-i & 0 \\
0 & \frac{1}{3}
\end{pmatrix}$ and $N_u(J)=\{ ( 0,0,0,0,\lambda ) \}_{\lambda \in \mathbb{C}}.$ \ Bearing this in mind, let us write $$J'=\begin{pmatrix}
-i & 0 & 0 & 0 & 0 \\
0 & \frac{1}{3} & 0 & 0 & 0 \\
0 & 0 & 0 & 1 & 0 \\
0 & 0 & 0 & 0 & 1 \\
0 & 0 & 0 & 0 & 0
\end{pmatrix}.$$ Adding a general element of $N(J)$ in the non-zero columns of $J'$ we obtain $$J''= \begin{pmatrix}
-i & 0 & 0 & 0 & 0 \\
0 & \frac{1}{3} & 0 & 0 & 0 \\
0 & 0 & 0 & 1 & 0 \\
0 & 0 & 0 & 0 & 1 \\
\alpha'_{51} & \alpha'_{52} & 0 & \alpha'_{54} & \alpha'_{55}
\end{pmatrix}.$$

 Now, we shall complete the zero columns of $J''$ with arbitrary parameters, except for the first $r=2$ rows of these columns. Notice that the zero column of $J''$ is the $j-$column with $$j\in \{r+1\}=\{3\} .$$ Hence, $$\tilde{J}= \begin{pmatrix}
-i & 0 & 0 & 0 & 0 \\
0 & \frac{1}{3} & 0 & 0 & 0 \\
0 & 0 & \alpha'_{33} & 1 & 0 \\
0 & 0 & \alpha'_{43} & 0 & 1 \\
\alpha'_{51} & \alpha'_{52} & 0 & \alpha'_{54} & \alpha'_{55}
\end{pmatrix}\, .$$
Firstly, let us explicitly give the set of GD1 inverses.
 Adding a general element of $N(J)$ in the non-zero columns of $J'$ except for the first 2 ones (recall in this case $r=2$ as $C\in \Mat_{2\times 2}(\C)$), we obtain $$J''= \begin{pmatrix}
-i & 0 & 0 & 0 & 0 \\
0 & \frac{1}{3} & 0 & 0 & 0 \\
0 & 0 & 0 & 1 & 0 \\
0 & 0 & 0 & 0 & 1 \\
0 & 0 & 0 & \alpha'_{45} & \alpha'_{55}
\end{pmatrix}.$$

 Now, we shall complete the zero columns of $J''$ with arbitrary parameters, except for the $i-$row of this column with $$i\in \{ r+m_{1}\}=\{5\}. $$ Notice that the zero column of $J''$ is the $j-$column with $$j\in \{r+1\}=\{3\} .$$ Hence, $$\tilde{J}= \begin{pmatrix}
-i & 0 & \alpha'_{13} & 0 & 0 \\
0 & \frac{1}{3} & \alpha'_{23} & 0 & 0 \\
0 & 0 & \alpha'_{33} & 1 & 0 \\
0 & 0 & \alpha'_{43} & 0 & 1 \\
0 & 0 & 0 & \alpha'_{54} & \alpha'_{55}
\end{pmatrix}\, .$$

 Let us denote $J^{GD1}=(\alpha_{ij})$ and, maintaining the notation, $\tilde{J}=(\alpha_{ij}').$ The only entry of $J^{GD1}$ left to compute is
$$\alpha_{53}=\begin{pmatrix}
0 & 0 & 0 & \alpha'_{54} & \alpha'_{55}
\end{pmatrix}\cdot \begin{pmatrix}
i & 0 & 0 & 0 & 0 \\
0 & 3 & 0 & 0 & 0 \\
0 & 0 & 0 & 0 & 0 \\
0 & 0 & 1 & 0 & 0 \\
0 & 0 & 0 & 1 & 0
\end{pmatrix} \cdot \left(\begin{matrix}
\alpha'_{31} \\
\alpha'_{32} \\
\alpha'_{33} \\
\alpha'_{43} \\
0
\end{matrix}\right)=\alpha'_{33}\cdot \alpha'_{54}+\alpha'_{43}\cdot \alpha'_{55}.$$

Therefore, $$\hat{J}=\begin{pmatrix}
-i & 0 & \alpha'_{13} & 0 & 0 \\
0 & \frac{1}{3} & \alpha'_{23} & 0 & 0 \\
0 & 0 & \alpha'_{33} & 1 & 0 \\
0 & 0 & \alpha'_{43} & 0 & 1 \\
0 & 0 & (\alpha'_{33}\cdot \alpha'_{54}+\alpha'_{43}\cdot \alpha'_{55}) & \alpha'_{54} & \alpha'_{55}
\end{pmatrix}. $$ To conclude, $A^{GD1}$ is a GD1 inverse of A if and only if: $$A^{GD1}= \begin{pmatrix}
1 & 0 & 3 & -1 & 1 \\
1 & -2 & 0 & 0 & 0 \\
3 & 1 & 2 & 0 & 6 \\
1 & 0 & -2 & 1 & 1 \\
0 & -1 & 0 & 0 & 1
\end{pmatrix} \begin{pmatrix}
-i & 0 & \alpha'_{13} & 0 & 0 \\
0 & \frac{1}{3} & \alpha'_{23} & 0 & 0 \\
0 & 0 & \alpha'_{33} & 1 & 0 \\
0 & 0 & \alpha'_{43} & 0 & 1 \\
0 & 0 & \alpha_{53} & \alpha'_{54} & \alpha'_{55}
\end{pmatrix}  \begin{pmatrix}
-4 & 3 & 2 & -4 & -4 \\
-2 & 1 & 1 & -2 & -2 \\
13 & -8 & -6 & 13 & 10 \\
32 & -20 & -15 & 33 & 25 \\
-2 & 1 & 1 & -2 & -1
\end{pmatrix},$$  with $\alpha_{13}',\alpha_{23}',\alpha_{33}',\alpha_{43}',\alpha_{54}',\alpha_{55}' \in \mathbb{C}$ and $\alpha_{53}$ being the one  previously calculated. Notice that in this case we have that $\text{dim N(A)}=1$ and $\mathrm{rk(A)}=4, \mathrm{rk(A_2)}=2$ and from Theorem \ref{T: Char AGD1 fin} we know that $A(GD1)\simeq \mathbb{C}^6$. Readers can easily check that this example is compatible with this result.
\end{exam}
\begin{exam}\label{Ex: Comp 1GD}
Similarly, let us offer the set of 1GD inverses of matrix $A$ stating carefully the different steps of the algorithm presented.\\
Again, we have that $C=\begin{pmatrix}
i & 0 \\
0 & 3
\end{pmatrix}\in Mat_{2 \times 2}(\mathbb{C})$ and $m_{1}=3$, such that the $(2+3)=(5)-$column of J is zero, $C^{-1}=\begin{pmatrix}
-i & 0 \\
0 & \frac{1}{3}
\end{pmatrix}$ and $N_u(J)=\{ ( 0,0,0,0,\lambda ) \}_{\lambda \in \mathbb{C}}.$  Bearing this in mind, let us settle again $$J'=\begin{pmatrix}
-i & 0 & 0 & 0 & 0 \\
0 & \frac{1}{3} & 0 & 0 & 0 \\
0 & 0 & 0 & 1 & 0 \\
0 & 0 & 0 & 0 & 1 \\
0 & 0 & 0 & 0 & 0
\end{pmatrix}.$$
 Let us denote $J^{1GD}=(\alpha_{ij})$ and, maintaining the notation, $\tilde{J}=(\alpha_{ij}').$ The only entry of $J^{1GD}$ left to compute is
$$\alpha_{53}=\begin{pmatrix}
\alpha'_{51} & \alpha'_{52} & 0 & \alpha'_{54} & \alpha'_{55}
\end{pmatrix}\cdot \begin{pmatrix}
i & 0 & 0 & 0 & 0 \\
0 & 3 & 0 & 0 & 0 \\
0 & 0 & 0 & 0 & 0 \\
0 & 0 & 1 & 0 & 0 \\
0 & 0 & 0 & 1 & 0
\end{pmatrix} \cdot \left(\begin{matrix}
0 \\
0 \\
\alpha'_{33} \\
\alpha'_{43} \\
0
\end{matrix}\right)=\alpha'_{53}\cdot \alpha'_{33}+\alpha'_{55}\cdot \alpha'_{43}.$$

Therefore, $$\hat{J}=\begin{pmatrix}
-i & 0 & 0 & 0 & 0 \\
0 & \frac{1}{3} & 0 & 0 & 0 \\
0 & 0 & \alpha'_{33} & 1 & 0 \\
0 & 0 & \alpha'_{43} & 0 & 1 \\
\alpha'_{51} & \alpha'_{52} & (\alpha'_{53}\cdot \alpha'_{33}+\alpha'_{55}\cdot \alpha'_{43}) & \alpha'_{54} & \alpha'_{55}
\end{pmatrix}. $$ To conclude, $A^{1GD}$ is a 1GD inverse of $A$ if and only if: $$A^{1GD}= \begin{pmatrix}
1 & 0 & 3 & -1 & 1 \\
1 & -2 & 0 & 0 & 0 \\
3 & 1 & 2 & 0 & 6 \\
1 & 0 & -2 & 1 & 1 \\
0 & -1 & 0 & 0 & 1
\end{pmatrix} \begin{pmatrix}
-i & 0 & 0 & 0 & 0 \\
0 & \frac{1}{3} & 0 & 0 & 0 \\
0 & 0 & 0 & 1 & 0 \\
0 & 0 & 0 & 0 & 1 \\
\alpha'_{51} & \alpha'_{52} & \alpha_{53} & \alpha'_{54} & \alpha'_{55}
\end{pmatrix}  \begin{pmatrix}
-4 & 3 & 2 & -4 & -4 \\
-2 & 1 & 1 & -2 & -2 \\
13 & -8 & -6 & 13 & 10 \\
32 & -20 & -15 & 33 & 25 \\
-2 & 1 & 1 & -2 & -1
\end{pmatrix},$$  with $\alpha_{51}',\alpha_{52}',\alpha_{54}',\alpha_{55}',\alpha_{33}',\alpha_{43}' \in \mathbb{C}$ and $\alpha_{53}$ being the one  previously calculated. Notice that in this case we have that $\text{dim N(A)}=1$ and $\mathrm{rk(A)}=4, \mathrm{rk(A_1)}=2$ and from Theorem \ref{T: Char AGD1 fin} we know that $A(1GD)\simeq \mathbb{C}^6$. Readers can easily check that this example is compatible with this result.
\end{exam}
 
\begin{rem}
It is worth noting that the presented algorithm is coherent with the one described in \cite[Section 5]{Die-Fpa}, which exposes the method to calculate the set of reflexive generalized inverses of a square matrix.
\end{rem}

\section{Binary relations with GD1 and 1GD inverses of finite potent endomorphisms}\label{S: Bin Rel}

In the present section we will study the binary relations that were introduced in \cite[Section 2.1, Section 3.1]{GD1N} for GD1 and 1GD inverses in the context of arbitrary vector spaces using finite potent endomorphisms. 

Let us start by recalling \cite[Lemma 3.1]{MinusDaa}. 
 
\begin{lem}\label{L: Gamma Reflex}
Let $\f\in \ed_k(V)$ be a linear operator over an arbitrary $k-$vector space. Then, the following morphism of sets $$\begin{array}{rccl}
\Gamma^{\f} \colon & X_{\f}(1)\times X_{\f}(1) & \rightarrow & X_{\f}(1,2) \\
& (\f_1^-,\f_2^-) & \mapsto & \f_1^-\circ \f \circ \f_2^-
\end{array}$$ is surjective.
\end{lem}

\begin{lem}\label{L: GD1GD1 en GD1}
Let $\f \in \ed_k(V)$ be a linear operator over an arbitrary $k-$vector space. Then, the following morphism of sets  $$\begin{array}{rccl}
\hat{\Gamma}^{\f} \colon & X_{\f}(GD1)\times X_{\f}(GD1)  & \rightarrow & X_{\f}(GD1) \smallskip\\
& (\f^{GD1},\tilde{\f}^{GD1}) & \mapsto & \f^{GD1}\circ \f \circ \tilde{\f}^{GD1}.
\end{array} $$
is surjective.
\end{lem}
\begin{proof}
The map of sets is well defined. To wit, let us suppose that $\f^{GD1}=\f^{GD}\circ \f \circ \f^-$ and that $\tilde{\f}^{GD1}=\tilde{\f}^{GD}\circ \f \circ \tilde{\f}^{-}$ for certain $\f^{GD},\tilde{\f}^{GD}\in X_{\f}(GD)$ and $\f^-,\tilde{\f}^-\in X_{\f}(1).$ Then $$\f^{GD1}\circ \f \circ \tilde{\f}^{GD1}=(\f^{GD}\circ \f \circ \f^-)\circ \f \circ (\tilde{\f}^{GD}\circ \f \circ \tilde{\f}^{-})=\f^{GD}\circ \f \circ \tilde{\f}^- $$ which is another $\widehat{\f}^{GD1}$ in virtue of Proposition \ref{T: Charact Algebraica}. Moreover, the set of maps is surjective. Given any $\f^{GD1}\in X_{\f}(GD1),$ it is clear that $\hat{\Gamma}^{\f}(\f^{GD1},\f^{GD1})=\f^{GD1}$ and the claim is proved.
\end{proof}

The reader can notice that we can prove the analogous result for 1GD inverses.

\begin{lem}\label{L: 1GD1GD en 1GD}
Let $\f \in \ed_k(V)$ be a linear operator over an arbitrary $k-$vector space. Then, the following map of sets  $$\begin{array}{rccl}
\hat{\Gamma}^{\f} \colon & X_{\f}(1GD)\times X_{\f}(1GD)  & \rightarrow & X_{\f}(1GD) \smallskip\\
& (\f^{1GD},\tilde{\f}^{1GD}) & \mapsto & \f^{1GD}\circ \f \circ \tilde{\f}^{1GD}.
\end{array} $$
is surjective.
\end{lem}

\subsection{GD1 and 1GD binary relations for finite potent operators}\label{ss: GD1o y 1GDo}
We shall define two binary relations that use GD1 inverses and 1GD inverses in their definitions in the context of arbitrary vector spaces, namely, infinite dimensional ones, in the framework of finite potent endomorphisms. Moreover, we shall prove that these relations are indeed partial orders in the set of finite potent endomorphisms over an arbitrary vector space.

\begin{lem}\label{L: Maximalidad Sub Inv}
Let $\f\in \ed_k(V)$ be a finite potent endomorphism and let $V=W_{\f}\oplus U_{\f}$ be the AST decomposition it induces. If $S\subseteq V$ is a $\f-$invariant subspace then: \begin{itemize}
\item if $\f_{\vert_{S}}\in \aut_{k}(S)$ then $S$ is finite dimensional and $S\subseteq W_{\f};$
\item if $\f_{\vert_{S}}$ is nilpotent then $S\subseteq U_{\f}.$
\end{itemize}
\end{lem}
\begin{proof}
Let us suppose that $S$ is a $\f-$invariant subspace such that $\f_{\vert_{S}}\in \aut_{k}(S).$ Then, for any $s\in S,$ there exists some $s'\in S$ such that $s=\f^m(s')$ for $m=i(\f).$ Hence, $s\in \Img(\f^m)=W_{\f}.$ Therefore, $S\subseteq W_{\f}$ and $S$ is finite dimensional. The equality holds precisely when $S+U_{\f}=V.$ For the second statement, the reasoning is analogous.
\end{proof}

\begin{defn}\label{D: Def GD1 order}
Let us consider two finite potent endomorphisms $\f, \g \in \ed_k(V).$ We will say that $\f$ is below $\g$ under the relation ``$\, \leq^{GD1} \,$'' if $\f \circ \hat{\f}^{GD1}=\g \circ \hat{\f}^{GD1}$ and $\tilde{\f}^{GD1}\circ \f=\tilde{\f}^{GD1}\circ \g$ for certain $\hat{\f}^{GD1},\tilde{\f}^{GD1}\in X_{\f}(GD1).$ 
\end{defn}

\begin{lem}\label{L: GD1 is G-Based}
Let $\f,\g \in \ed_k(V)$ be two finite potent endomorphisms. Then, $\f \, \leq^{GD1} \, \g$ (in the sense of Definition \ref{D: Def GD1 order}) if and only if there exists a $\f^{GD1}\in X_{\f}(GD1)$ such that $\f\circ \f^{GD1}=\g\circ \f^{GD1}$ and $\f^{GD1}\circ \f=\f^{GD1}\circ \g .$
\end{lem}
\begin{proof}
It is a direct consequence of Lemma \ref{L: GD1GD1 en GD1}.
\end{proof}

\begin{defn}\label{D: Def 1GD order}
Let us consider two finite potent endomorphisms $\f, \g \in \ed_k(V).$ We will say that $\f$ is below $\g$ under the relation ``$\, \leq^{1GD} \,$'' if $\f \circ \hat{\f}^{1GD}=\g \circ \hat{\f}^{1GD}$ and $\tilde{\f}^{1GD}\circ \f=\tilde{\f}^{1GD}\circ \g$ for certain $\hat{\f}^{GD1},\tilde{\f}^{1GD}\in X_{\f}(1GD).$ 
\end{defn}

\begin{lem}\label{L: 1GD is G-Based}
Let $\f,\g \in \ed_k(V)$ be two finite potent endomorphisms. Then, $\f \, \leq^{1GD} \, \g$ (in the sense of Definition \ref{D: Def 1GD order}) if and only if there exists a $\f^{1GD}\in X_{\f}(1GD)$ such that $\f\circ \f^{1GD}=\g\circ \f^{1GD}$ and $\f^{1GD}\circ \f=\f^{1GD}\circ \g .$
\end{lem}
\begin{proof}
It follows directly from Lemma \ref{L: 1GD1GD en 1GD}.
\end{proof}

\begin{rem}
Lemma \ref{L: GD1 is G-Based} and Lemma \ref{L: 1GD is G-Based} state that the relations ``$\, \leq^{GD1} \,$'' and ``$\, \leq^{1GD} \,$'' are $\mathcal{G}-$based orders (recall Definition \ref{D: G-based rel}).
\end{rem}

Henceforth, whenever the relations ``$\, \leq^{GD1} \,$'' and ``$\, \leq^{1GD} \,$'' appear we will use its $\mathcal{G}-$based form.

\begin{lem}\label{L: GD1o 1GDo are inv}
Let $\f, \g \in \ed_k(V)$ be two finite potent endomorphisms and let $\tau \in \mathrm{Aut}_k(V)$ be any automorphism. Then: \begin{itemize}
\item $\f \, \leq^{GD1} \, \g$ if and only if $\tau \circ \f\circ \tau^{-1} \, \leq^{GD1} \, \tau \circ\g\circ \tau^{-1};$
\item $\f \, \leq^{1GD} \, \g$ if and only if $\tau \circ \f\circ \tau^{-1} \, \leq^{1GD} \, \tau \circ\g\circ \tau^{-1};$
\end{itemize}
\end{lem}
\begin{proof}
It is a direct consequence of Lemma \ref{L: GD1 inv conj} and Lemma \ref{L: 1GD inv conj}.
\end{proof}

\begin{thm}\label{T:3 equiv  }
Let $\f,\g \in \ed_k(V)$ be two finite potent endomorphisms. Then, the following statements are equivalent:
\begin{itemize}
\item[I.)] $\f \, \leq^{GD1} \, \g$  ;
\item[II.)] $\f=\f\circ \f^{-}\circ \g=\g \circ \f^{GD}\circ \f$ for certain $\f^{-}\in X_{\f}(1)$ and $\f^{GD}\in X_{\f}(GD);$  
\item[III.)] $\f=\f\circ \f^{GD1}\circ \g=\g\circ \f^{GD1}\circ \f$ for certain $\f^{GD1}\in X_{\f}(GD1).$
\end{itemize}
\end{thm}
\begin{proof}
$I) \Longrightarrow II).$ As $\f \circ \f^{GD1}=\g\circ \f^{GD1}$ and $\f^{GD1}\circ \f =\f^{GD1}\circ \g $ for certain $\f^{GD1}\in X_{\f}(GD1).$ In virtue of the expression of any GD1 inverse obtained in Proposition \ref{T: Charact Algebraica}, we know that $\f\circ \f^{GD}\circ \f \circ \f^{-}=\g\circ\f^{GD}\circ \f \circ \f^{-} $ and therefore \begin{equation}\label{eq: 3equiv a}
\f\circ \f^{-}=\g\circ \f^{GD}\circ \f \circ \f^{-}.
\end{equation}
Similarly, $\f^{GD}\circ \f \circ \f^{-}\circ \f =\f^{GD}\circ \f \circ \f^{-}\circ \g ,$ so \begin{equation}\label{eq: 3equiv b}
\f^{GD}\circ \f=\f^{GD}\circ \f \circ \f^{-}\circ \g .
\end{equation}
Composing in \eqref{eq: 3equiv a} with $\f$ on the right hand side and in \eqref{eq: 3equiv b} with $\f$ on the left hand side, one obtains that: $\f=\g\circ \f^{GD}\circ \f$ and $\f=\f\circ\f^{-}\circ \g$ respectively, so the first implication is proved. \\
$II) \Longrightarrow III).$ Let us suppose that $\f=\f\circ \f^{-}\circ \g=\g \circ \f^{GD}\circ \f$ holds for certain $\f^{GD}\in X_{\f}(GD)$ and $\f^{-}\in X_{\f}(1).$ Take the composition: $\f^{GD}\circ \f \circ \f^{-}$ which is a GD1 inverse of $\f .$ By direct computation, one gets that: \begin{align*}
\f\circ (\f^{GD}\circ \f \circ \f^{-})\circ \g & = \f \circ \f^{-}\circ \g =\f ;\\
\g\circ (\f^{GD}\circ \f \circ \f^{-})\circ \f & =\g \circ \f^{GD}\circ \f=\f ,
\end{align*}
so the condition in the statement is deduced.\\
$III)\Longrightarrow I).$ Let us now suppose that $\f=\f\circ \f^{GD1}\circ \g=\g\circ \f^{GD1}\circ \f $ for certain $\f^{GD1}.$ To wit, right-composing with $\f^{GD1}$ in the first equality one gets: $$\f\circ \f^{GD1}=\g \circ \f^{GD1}\circ \f \circ \f^{GD1}=\g \circ \f^{GD1} $$ 
as GD1 inverses are, in particular, $\{2 \}-$inverses (recall Lemma \ref{L: GD1 son reflex}). Reasoning on an analogous way, one obtains that: $$\f^{GD1}\circ \f=\f^{GD1}\circ \f \circ \f^{GD1}\circ \g =\f^{GD1}\circ \g $$ so both conditions of the relation $\leq^{GD1}$ hold and hence the claim is proved.
\end{proof}

As we shall see later in Theorem \ref{T: GD1 Partial Order}, item $III)$ of this theorem proves that the relation $\leq^{GD1}$ is antisymmetric. However, transitivity is not directly granted. We shall devote the rest of this part of the article to seek a characterization of the aforementioned relation which serves to show that it is transitive and hence, a partial order in the set of finite potent endomorphisms.

The following result can be proven to be general for any $\mathcal{G}$-based relation defined using any subset of $\{1\}-$inverses on its definition, however, we adapt it for the binary relation we have introduced using $GD1-$inverses to maintain the coherence in exposition.

\begin{cor}\label{C: GD1 Implica Space y Minus}
Let $\f , \g \in \ed_k(V)$ be two finite potent endomorphisms such that $\f \, \leq^{GD1} \, \g ,$ then $\f \, \leq^{-} \, \g$ (Definition \ref{D: Minus FP}). Moreover, $\f \, <^s \g $ (Definition \ref{D: Space Pre-Ord}). 
\end{cor}
\begin{proof}
Let us suppose that $\f \circ \f^{GD1}=\g \circ \f^{GD1}$ and $\f^{GD1}\circ \f =\f^{GD1}\circ \g$ for certain $\f^{GD1}\in X_{\f}(GD1).$ As $X_{\f}(GD1)\subseteq X_{\f}(1)$ (in fact $X_{\f}(GD1)\subseteq X_{\f}(1,2)$) then condition $I)$ of Theorem \ref{T: Equiv Minus FP} (and also II') of Theorem \ref{T: Equiv Minus FP}) are satisfied and hence $\f \, \leq^{-} \, \g .$ The statement dealing with the space pre-order (Definition \ref{D: Space Pre-Ord}) is directly deduced from $II)$ or $III)$ of Theorem \ref{T:3 equiv  }. Namely, let us suppose that $\f \, \leq^{GD1} \, \g$ and hence $\f=\f\circ \f^{-}\circ \g=\g \circ \f^{GD}\circ \f$ for certain $\f^{-}\in X_{\f}(1)$ and $\f^{GD}\in X_{\f}(GD).$ As $\f=\g \circ \f^{GD}\circ \f$ it is clear that $\Img(\f)\subseteq \Img(\g)$ and, similarly, using that $\f=\f\circ \f^{-}\circ \g$ one deduces that $\Ker(\g)\subseteq \Ker(\f),$ so both conditions of Definition \ref{D: Space Pre-Ord} hold and the result is proved. 
\end{proof}

\begin{rem}\label{R: Sin Contencion}
Before continuing, let us use Theorem \ref{T:3 equiv  } to show that the relation given in Definition \ref{D: Def GD1 order} for two finite potent operators $\f, \g \in \ed_k(V)$ can not be characterized in terms of the sets of GD1 inverses of $\f$ and $\g.$ Precisely, we shall remark that $\f \, \leq^{GD1} \, \g$ does not imply that $X_{\g}(GD1)\subseteq X_{\f}(GD1).$ Note that proving that a binary relation defined as the one in Definition \ref{D: Def GD1 order} which is proven to be equivalent to an inclusion of certain subset of $\{1\}-$inverses is immediately transitive. In order to do so, let us consider the following two matrices: $$ A=\left(\begin{matrix}
2 & 0 \\
0 & 0
\end{matrix}\right)
\text{ and }
B=\left(\begin{matrix}
2 & -6 \\
0 & 3
\end{matrix}\right).$$
Applying the algorithm described in Section \ref{ss: Computation GD1 Matrices} to $A,$ one gets that the set of GD1 inverses of $A$ is precisely: $$A(GD1)= \bigg\{\left(\begin{matrix}
\frac{1}{2} & a \\
0 & 0
\end{matrix}\right):a\in k\bigg\}. $$  Meanwhile, $B$ is an invertible matrix and the set of GD1 inverses is precisely $B^{-1}$: $$B(GD1)=\{B^{-1}\}=\bigg\{\left(\begin{matrix}
\frac{1}{2} & 1 \\
0 & \frac{1}{3}
\end{matrix}\right)\bigg\}. $$ Clearly, $B^{GD1}\notin A(GD1)$ and, considering, for example, a fixed $A^{GD1}$ such that $A^{GD1}=\left(\begin{matrix}
\frac{1}{2} & 1 \\
0 & 0
\end{matrix}\right)\in A(GD1)$ one has that the conditions of Theorem \ref{T:3 equiv  } are satisfied for $A^{GD1}.$ Namely: \begin{align*}
A\cdot A^{GD1}\cdot B & = \left(\begin{matrix}
2 & 0 \\
0 & 0
\end{matrix}\right)\cdot \left(\begin{matrix}
\frac{1}{2} & 1 \\
0 & 0
\end{matrix}\right)\cdot \left(\begin{matrix}
2 & -6 \\
0 & 3
\end{matrix}\right)= \left(\begin{matrix}
2 & 0 \\
0 & 0
\end{matrix}\right).   \smallskip\\
B\cdot A^{GD1}\cdot A & = \left(\begin{matrix}
2 & -6 \\
0 & 3
\end{matrix}\right)\cdot \left(\begin{matrix}
\frac{1}{2} & 1 \\
0 & 0
\end{matrix}\right)\cdot \left(\begin{matrix}
2 & 0 \\
0 & 0
\end{matrix}\right)= \left(\begin{matrix}
2 & 0 \\
0 & 0
\end{matrix}\right).
\end{align*} 
Therefore, we have shown that $A\leq^{GD1}B$ but $B(GD1)\nsubseteq A(GD1).$ Bearing in mind the well-known relationship between matrices and endomorphisms over finite dimensional vector spaces, which are an example of finite potent endomorphisms, we conclude.
\end{rem}

\begin{prop}\label{P: Charact GD1ord en Wf}
Let $\f, \g \in \ed_k(V)$ be two finite potent endomorphisms and let $V=W_{\f}\oplus U_{\f}$ be the AST decomposition $\f$ induces. Then, 
$$\f_{\vert_{W_{\f}}}=\g_{\vert_{W_{\f}}} \text{ if and only if } 
(\f \circ \f^{GD1})_{\vert_{W_{\f}}} =(\g \circ \f^{GD1})_{\vert_{W_{\f}}}$$
 for any $\f^{GD1}\in X_{\f}(GD1).$
\end{prop}
\begin{proof}
To prove the direct we must show that $(\g \circ \f^{GD1})_{\vert_{W_{\f}}}=\Id_{\vert_{W_{\f}}},$ because $(\f \circ \f^{GD1})_{\vert_{W_{\f}}}=(\f\circ \f^-)_{\vert_{W_{\f}}}=\Id_{\vert_{W_{\f}}}$ for any $\f^{GD1}\in X_{\f}(GD1).$ Moreover, by Theorem \ref{P: CharactGD1FP}, one has that $$(\g\circ \f^{GD}\circ \f \circ \f^-)_{\vert_{W_{\f}}}=(\g\circ \f^{GD})_{\vert_{W_{\f}}}=(\g\circ (\f_{\vert_{W_{\f}}})^{-1})_{\vert_{W_{\f}}}=(\g\circ (\g_{\vert_{W_{\f}}})^{-1})_{\vert_{W_{\f}}} $$ by hypothesis, so we conclude. For the converse, $$(\f \circ \f^{GD1})_{\vert_{W_{\f}}}=\Id_{\vert_{W_{\f}}}=(\g\circ \f^{GD1})_{\vert_{W_{\f}}}.$$ Therefore, one gets that $$(\g \circ (\f_{\vert_{W_{\f}}})^{-1})_{\vert_{W_{\f}}}=\Id_{\vert_{W_{\f}}} $$ and we conclude by the uniqueness of the inverse of a linear operator over a finite dimensional vector space. 
\end{proof}

\begin{cor}\label{C: fyg coinciden en W GD1}
Let us consider finite potent endomorphisms $\f,\g \in \ed_k(V)$ with $i(\f)=m.$ If $\f \, \leq^{GD1} \, \g,$ then 
 $\f_{\vert_{W_{\f}}}=\g_{\vert_{W_{\f}}}.$ In particular, $\g$ leaves $W_{\f}$ $\g-$invariant and $\f^{m+1}=\g\circ \f^m,$ so $W_{\f}\subseteq \Img(\g).$
\end{cor}
\begin{proof}
It is deduced from Definition \ref{D: Def GD1 order} and Proposition \ref{P: Charact GD1ord en Wf}. The rest of facts are a consequence of the construction and properties of the AST decomposition.
\end{proof}

\begin{cor}\label{C: Wf en Wg}
If $\f, \g \in \ed_k(V)$ are two finite potent endomorphisms such that $\f \, \leq^{GD1} \, \g$ then $W_{\f}\subseteq W_{\g}.$ Moreover, $$W_{\g}=W_{\f}\oplus (U_{\f}\cap W_{\g}). $$
\end{cor}
\begin{proof}
It is a direct consequence of Lemma \ref{L: Maximalidad Sub Inv} and Corollary \ref{C: fyg coinciden en W GD1} applied to the endomorphism $\g$. To see the second statement, we shall prove the two following conditions $W_{\f}+(U_{\f}\cap W_{\g})=W_{\g}$ and $W_{\f}\cap (U_{\f}\cap W_{\g})=\{ 0 \}.$ Let $\bar{w}\in W_{\g},$ using the AST decomposition induced by $\f,$ we  write $\bar{w}=w+u$ with $w\in W_{\f}$ and $u\in U_{\f}.$ Therefore, $u=\bar{w}-w\in U_{\f}\cap W_{\g}$ because $W_{\f}\subseteq W_{\g}$ by the recently proven result. Finally, $W_{\f}\cap (U_{\f}\cap W_{\g})\subseteq W_{\f}\cap U_{\f                                                                                                                                      }=\{ 0\},$ and the claim is deduced.
\end{proof}

\begin{lem}\label{L: gGDffMenos}
Let us consider two finite potent endomorphisms $\f , \g \in \ed_k(V)$ such that $\f \, \leq^{GD1}\, \g .$ Then, the following morphism of sets $$\begin{array}{rccl}
\Gamma \colon & X_{\g}(GD)\times X_{\f}(1)  & \rightarrow & X_{\f}(GD1)\smallskip\\
& (\g^{GD},\f^{-}) & \mapsto & \g^{GD}\circ \f \circ \f^{-}
\end{array} $$ is well defined for any $\g^{GD}\in X_{\g}(GD)$ and $\f^-\in X_{\f}(1).$
\end{lem}
\begin{proof}
We shall prove that the endomorphism $\phi = \g^{GD}\circ \f \circ \f^{-},$ for any $\g^{GD}\in X_{\g}(GD1)$ and $\f^-\in X_{\f}(1),$ is a reflexive generalized inverse of $\f$ that leaves $W_{\f}$ invariant, in virtue of the characterization of GD1-inverses given in Theorem \ref{T: Caract GD1 INV AST}. If $\f \, \leq^{GD1} \, \g ,$ then, by Corollaries \ref{C: fyg coinciden en W GD1} and \ref{C: Wf en Wg}, we know that $\f_{\vert_{W_{\f}}}=\g_{\vert_{W_{\f}}}$ and $W_{\g}=W_{\f}\oplus (U_{\f}\cap W_{\g}).$ In virtue of the characterization of G-Drazin inverses of a finite potent endomorphism using the AST decomposition, recall Proposition \ref{P: Charact GDrazinfp}, we know that any $\g^{GD}\in X_{\g}(GD)$ satisfies that $(\g^{GD})_{\vert_{W_{\g}}}=(\g_{\vert_{W_{\g}}})^{-1}.$ In particular, we are only interested on the fact that, by Corollary \ref{C: fyg coinciden en W GD1}  \begin{equation}\label{eq: gGD en Wf}
(\g^{GD})_{\vert_{W_{\f}}}=(\f_{\vert_{W_{\f}}})^{-1}.
\end{equation}
The claim is now that the endomorphism $\phi$ is a reflexive generalized inverse of $\f .$ To wit, as $\f \, \leq^{GD1} \, \g$ we know that $\f \, \leq^{-} \, \g .$ By Remark \ref{R: MinusCoinciden} and Theorem \ref{T: Equiv Minus FP} we know that $X_{\g}(1)\subseteq X_{\f}(1).$ In particular, $X_{\g}(GD)\subseteq X_{\g}(1)\subseteq X_{\f}(1).$ Therefore, the claim follows precisely from Lemma \ref{L: Gamma Reflex}. We shall check that $\phi$ leaves $W_{\f}$ invariant to conclude. This can be done by direct computation. To wit, $$\phi (W_{\f})=(\g^{GD}\circ \f \circ \f^{-})(W_{\f})=(\g^{GD})(W_{\f})=(\f_{\vert_{W_{\f}}})^{-1}(W_{\f})\subseteq W_{\f} $$ as it follows from the discussion showed in \eqref{eq: gGD en Wf}. Thus, we conclude that $\phi$ leaves $W_{\f}$ invariant and both conditions of Theorem \ref{T: Caract GD1 INV AST} are satisfied. Hence, $\phi \in X_{\f}(GD1)$ for any $\g^{GD}\in X_{\g}(GD1)$ and $\f^-\in X_{\f}(1).$
\end{proof}

\begin{thm}\label{T: Charact GD1 Order FP}
Let $\f, \g \in \ed_k(V)$ be two finite potent endomorphisms. Then, $\f \leq^{GD1}\g$ if and only if \begin{itemize}
\item[I.)] $\f_{\vert_{W_{\f}}}=\g_{\vert_{W_{\f}}},$ \\
\item[II.)] $\f\leq^{-}\g$ (Definition \ref{D: Minus FP}).
\end{itemize}
\end{thm}
\begin{proof}
Firstly, let us suppose that $\f \, \leq^{GD1} \, \g .$ Then, the fact that $I)$ holds is precisely Proposition \ref{P: Charact GD1ord en Wf} (or Corollary \ref{C: Wf en Wg}). Indeed, $II)$ was shown in Corollary \ref{C: GD1 Implica Space y Minus}. Let us prove the converse. As $\f \, \leq^{-} \, \g ,$ using Theorem \ref{T: Equiv Minus FP}, let us consider certain $\f^- \in X_{\f}(1)$ such that \begin{equation}\label{eq: condminusdem}
\f \circ \f^- =\g \circ \f^- \text{ and } \f^-\circ \f=\f^- \circ \g. 
\end{equation} Now, define the following endomorphism: $$\phi= \g^{GD}\circ \f \circ \f^{-} $$ for certain $\g^{GD}\in X_{\g}(GD).$ By Lemma \ref{L: gGDffMenos} we know that $\phi \in X_{\f}(GD1).$ We must check that the equalities of Definition \ref{D: Def GD1 order} hold for this $\phi .$ To wit, using that, in particular, $\g^{GD}\in X_{\f}(1)$ (the reasoning was done in the proof of Lemma \ref{L: gGDffMenos}) and \eqref{eq: condminusdem} one has that: \begin{align*}
\f \circ \phi & = \f \circ (\g^{GD}\circ \f \circ \f^{-})= \f \circ \f^- = \g \circ \f^{-}= \g \circ \g^{GD} \circ \g \circ \f^{-}= \\& =\g \circ \g^{GD}\circ \f \circ \f^-=\g \circ \phi, ; \\
\phi \circ \f & =(\g^{GD}\circ \f \circ \f^{-})\circ \f=\g^{GD}\circ \f \circ \f^- \circ \g=\phi \circ \g.
\end{align*} Hence, Definition \ref{D: Def GD1 order} is satisfied and we conclude the proof.
\end{proof}

\begin{lem}\label{L: Transitivity}
Let $\f, \g, \gamma \in \ed_k(V)$ be three finite potent endomorphisms and let $V=W_{\f}\oplus U_{\f}=W_{\g}\oplus U_{\g}=W_{\gamma}\oplus U_{\gamma}$ be the respective AST decompositions they induce. If $\f \, \leq^{GD1} \, \g $ and $\g \, \leq^{GD1} \, \gamma$ then $$\f_{\vert_{W_{\f}}}=\gamma_{\vert_{W_{\f}}}.$$
\end{lem}
\begin{proof}
As $\f \, \leq^{GD1} \, \g $ and $\g \, \leq^{GD1} \, \gamma$ then, by Corollaries \ref{C: fyg coinciden en W GD1} and \ref{C: Wf en Wg} , we deduce that $W_{\f}\subseteq W_{\g}\subseteq W_{\gamma}$ with $\f_{\vert_{W_{\f}}}=\g_{\vert_{W_{\f}}}$ and $\g_{\vert_{W_{\g}}}=\gamma_{\vert_{W_{\g}}}.$ Hence, we conclude as desired $\f_{\vert_{W_{\f}}}=\g_{\vert_{W_{\f}}}=\gamma_{\vert_{W_{\f}}}.$
\end{proof}

\begin{thm}\label{T: GD1 Partial Order}
The binary relation ``$\, \leq^{GD1} \,$'' is a partial order on the set of finite potent endomorphisms over an arbitrary $k-$vector space.
\end{thm}
\begin{proof}
Reflexivity is direct. Let us prove antisymmetry. Let us suppose that $\f \, \leq^{GD1} \, \g$ and $\g \, \leq^{GD1} \, \f ,$ this is $\f\circ \f^{GD1}=\g\circ \f^{GD1} ; \f^{GD1}\circ \f=\f^{GD1}\circ \g$ and $\g \circ \g^{GD1}=\f\circ \g^{GD1}$ for certain $\f^{GD1}\in X_{\f}(GD1)$ and $\g^{GD1}\in X_{\g}(GD1).$ Then, by Theorem \ref{T:3 equiv  }, we can write: \begin{align*}
\f & = \f \circ \f^{GD1}\circ \g=\g \circ \f^{GD1}\circ \f \\
\g & = \g \circ \g^{GD1}\circ \f=\f \circ \g^{GD1}\circ \g.
\end{align*}
Hence, \begin{align*}
\f & =\f \circ \f^{GD1}\circ \g=\f \circ \f^{GD1}\circ \g \circ \g^{GD1}\circ \g= \f\circ \f^{GD1}\circ \f \circ \g^{GD1}\circ \g=\\& = \f \circ \g^{GD1}\circ \g=\g \circ \g^{GD1}\circ \g=\g.
\end{align*} To finish the proof, let us check transitivity. Let us consider three finite potent endomorphisms $\f, \g , \gamma \in \ed_k(V),$ such that $\f \, \leq^{GD1} \, \g$ and $\g \, \leq^{GD1} \, \gamma .$ We shall check that $\f \, \leq^{GD1} \, \gamma .$ By Lemma \ref{L: Transitivity} we already know that $\f_{\vert_{W_{\f}}}=\gamma_{\vert_{W_{\f}}}.$ In virtue of Theorem \ref{T: Charact GD1 Order FP} we know that $\f \, \leq^{-}\, \g$ and $\g \, \leq^{-} \, \gamma .$ Hence, by Theorem \ref{T: Minus is PO FP}, we can conclude that $\f \, \leq^{-}\, \gamma .$ Summing up, we have proven that $\f_{\vert_{W_{\f}}}=\gamma_{\vert_{W_{\f}}}$ and $\f \, \leq^{-}\, \gamma .$ Therefore, using again Theorem \ref{T: Charact GD1 Order FP} (in the other direction) we conclude that $\f \, \leq^{GD1} \, \gamma$ and the claim is proved.
\end{proof}

We will devote the rest of the section to prove that the binary relation ``$\leq^{1GD}$'' defines a partial order in the set of finite potent endomorphisms.
The following two results are analogous as the ones previously proved for the ``$\leq^{GD1}$'' binary relation, namely Theorem \ref{T:3 equiv  } and Corollary \ref{C: GD1 Implica Space y Minus} and therefore, their proofs are not included.

\begin{thm}\label{T: 3 equiv 1GD}
Let $\f, \g \in \ed_k(V)$ be two finite potent endomorphisms. Then, the following statements are equivalent: \begin{itemize}
\item[I.)]$\f \, \leq^{1GD}\, \g$
\item[II.)]$\f=\f\circ \f^{GD}\circ \g=\g\circ \f^-\circ \f$ for certain $\f^-\in X_{\f}(1)$ and $\f^{GD}\in X_{\f}(GD)$
\item[III.)]$\f=\f\circ \f^{1GD}\circ \g=\g\circ \f^{1GD}\circ \f$ for certain $\f^{1GD}\in X_{\f}(1GD).$
\end{itemize}
\end{thm} 

\begin{cor}\label{C: 1GD implica minus y space}
Let $\f , \g\in \ed_k(V)$ be finite potent endomorphisms such that $\f \, \leq^{1GD}\, \g,$ then $\f \, \leq^{-}\, \g $ (Definition \ref{D: Minus FP}). Moreover, $\f \, <^s \, \g$ (Definition \ref{D: Space Pre-Ord}).
\end{cor}

\begin{lem}\label{L: 1GD cancel m+1}
Let $\f \in \ed_k(V)$ be a finite potent endomorphism of index $i(\f)=m .$ Then $\f^{m+1}\circ \f^{1GD}=\f^m$ for any $\f^{1GD}\in X_{\f}(1GD).$
\end{lem}
\begin{proof}
If $V=W_{\f}\oplus U_{\f}$ is the AST decomposition $\f$ induces then one has that: \begin{align*}
(\f^{m+1}\circ \f^{1GD})_{\vert_{U_{\f}}} & =0=(\f^m)_{\vert_{U_{\f}}}; \\
(\f^m\circ \f \circ \f^{1GD})_{\vert_{W_{\f}}}&=(\f^m\circ \Id)_{\vert_{W_{\f}}}=(\f^m)_{\vert_{W_{\f}}},
\end{align*} for any $\f^{1GD}\in X_{\f}(1GD).$
\end{proof}

\begin{lem}\label{L: fm+1 fmg}
Let $\f , \g \in \ed_k(V)$ with $\f$ a finite potent endomorphism of index $i(\f)=m.$ If $\f^{1GD}\circ \g =\f^{1GD}\circ \f$ for any $\f^{1GD}\in X_{\f}(1GD)$ then $\f^{m+1}=\f^m \circ \g .$
\end{lem}
\begin{proof}
By direct calculation, one gets that: $$\f^m\circ \g=\f^{m+1}\circ \f^{1GD}\circ \g=\f^{m+1}\circ \f^{1GD}\circ\f =\f^{m+1} $$ as we wanted to show.
\end{proof}

\begin{lem}\label{L: Resta en Uf}
Let $\f , \g \in \ed_k(V)$ with $\f$ being a finite potent endomorphism of index $i(\f)=m.$ One has that $\f^{m+1}=\f^m \circ \g $ if and only if $\Img(\g - \f)\subseteq U_{\f}.$ 
As a consequence, under any one of both equivalent conditions, $\g = \f + \phi$ with $\phi \in \ed_k(V)$ satisfying that $\Img(\phi)\subseteq U_{\f}.$
\end{lem}
\begin{proof}
Notice that one can rewrite $\f^{m+1}=\f^m\circ \g$ as $\f^m\circ \f=\f^m\circ \g .$ Hence $\f^m \circ (\g-\f)=0$ which is to say that $\Img(\g-\f)\subseteq \Ker(\f^m)=U_{\f}.$ Conversely, if $\Img(\g - \f)\subseteq U_{\f},$ then, for any $v\in V ,$ $\g(v)-\f(v)\in U_{\f}.$ Therefore, $\f^m (\g(v)-\f(v))=0,$ this is, $(\f^m \circ \g)(v)=(\f^m\circ \f)(v).$ This holds for any $v\in V,$ so the converse is deduced.
\end{proof}

\begin{prop}\label{P: Inv y Potencias}
Let $\f, \g \in \ed_k(V)$ with $\f$ being finite potent of index $i(\f)=m.$ If $\f^{m+1}=\f^m \circ \g $ then: \begin{itemize}
\item[I.)] $\g$ leaves $U_{\f}$ invariant.
\item[II.)] $\f^m\circ \f^s=\f^m\circ \g^s$ for any $s\geq 1 .$
\end{itemize}
\end{prop}
\begin{proof}
Let us begin proving $I.$ Let us consider $u\in U_{\f}.$ Then, one has that $0=\f^{m+1}(u)=\f^m(\g(u)) $ and hence, $\g(u)\in \Ker(\f^m)=U_{\f}.$ Let us continue by showing $II.$ The proof is done by using the hypothesis recursively, precisely: $$\f^m\circ \g^s=\f^{m+1}\circ \g^{s-1}=\dots =\f^{m+s}, $$ and the statement is proved.
\end{proof}

\begin{cor}\label{C: Ug en Uf}
Let us consider two finite potent endomorphisms $\f, \g \in \ed_k(V),$ with $i(\f)=m, i(\g)=s$ and such that they induce the following respective AST decompositions $V=W_{\f}\oplus U_{\f}=W_{\g}\oplus U_{\g}.$  If $\f^{m+1}=\f^m \circ \g $ then $U_{\g}\subseteq U_{\f}.$
\end{cor}
\begin{proof}
By Proposition \ref{P: Inv y Potencias} we know that in this conditions $\f^m\circ \f^s=\f^m\circ \g^s.$ For any $u\in U_{\g}$ we have that $$\f^{m+s}(u)=\f^m\circ \g^s(u)=0 $$ so $u\in \Ker(\f^{m+s})=\Ker(\f^m)=U_{\f}$ by definition of AST decomposition.
\end{proof}

\begin{prop}\label{P: Descomp Uf}
Let us consider two finite potent endomorphisms $\f, \g \in \ed_k(V),$  such that they induce the following respective AST decompositions $V=W_{\f}\oplus U_{\f}=W_{\g}\oplus U_{\g}.$ If $\f \, \leq^{1GD} \, \g$ then $U_{\g}\subseteq U_{\f}.$  Moreover, $U_{\f}=U_{\g}\oplus (W_{\g}\cap U_{\f}).$ 
\end{prop}
\begin{proof}
If $\f \, \leq^{1GD} \, \g$ we know that $U_{\g}\subseteq U_{\f}$ by Lemma \ref{L: fm+1 fmg} and Corollary \ref{C: Ug en Uf}.
We shall prove that $U_{\g}+(U_{\f}\cap W_{\g})=U_{\f}$ and $U_{\g}\cap (U_{\f}\cap W_{\g})=\{0\}.$ Let us express any $\bar{u}\in U_{\f}$ using the AST decomposition induced by $\g$ as $\bar{u}=w+u$ with $w\in W_{\g}$ and $u\in U_{\g}.$ Therefore, using that $U_{\g}\subseteq U_{\f},$ one has that $$w=\bar{u}-u\in U_{\f}\cap W_{\g}, $$ and the first statement is deduced. Finally, it is clear from the construction of the AST decomposition that $$U_{\g}\cap (U_{\f}\cap W_{\g})\subseteq U_{\g}\cap W_{\g}=\{0\}, $$ so we conclude. 
\end{proof}

\begin{lem}\label{L: fmenos f g gdraz}
Let us consider two finite potent endomorphisms $\f, \g \in \ed_k(V),$ with $i(\f)=m,$ such that $\f\, \leq^{-}\, \g$ (Definition \ref{D: Minus FP}) and such that $\f^{m+1}=\f^m\circ \g .$  Then, the following morphism of sets $$\begin{array}{rccl}
\Gamma \colon & X_{\f}(1)\times X_{\g}(GD)  & \rightarrow & X_{\f}(1GD)\smallskip\\
& (\f^{-},\g^{GD}) & \mapsto & \f^{-}\circ \f \circ \g^{GD}
\end{array} $$ is well defined for any $\g^{GD}\in X_{\g}(GD)$ and $\f^-\in X_{\f}(1).$
\end{lem}
\begin{proof}
As $\f \, \leq^{-} \, \g$ we know by Theorem \ref{T: Equiv Minus FP} that $X_{\g}(1)\subseteq X_{\f}(1).$ In particular, it is clear that $X_{\g}(GD)\subseteq X_{\f}(1).$ Therefore, we deduce that $\f^{-}\circ \f \circ \g^{GD}\in X_{\f}(1,2)$ for any $\f^-\in X_{\f}(1)$ and $\g \in X_{\g}(GD)$ in virtue of Lemma \ref{L: Gamma Reflex}. To conclude, using the characterization of 1GD inverses given in Theorem \ref{T: Charact 1GD AST e inva}, we shall prove that $\f^{-}\circ \f \circ \g^{GD}$ leaves $U_{\f}$ invariant. Now, on one hand, in these hypothesis we have that $U_{\g}\subseteq U_{\f}$ by Corollary \ref{C: Ug en Uf}. On the other, as any G-Drazin inverse of a finite potent endomorphism $\g$ leaves both, $W_{\g}$ and $U_{\g}$ invariant, one gets that $$\g^{GD}(U_{\g})\subseteq U_{\g}\subseteq U_{\f}, $$ so $(\f^-\circ \f \circ \g^{GD})(U_{\g})\subseteq U_{\f}$ as $(\f^-\circ \f)(U_{\f})\subseteq U_{\f}$ by the construction of the Jordan basis of $U_{\f}$ and the characterization of the $\{1\}-$inverses of a finite potent endomorphism over it. Hence, we conclude that $\f^-\circ \f \circ \g^{GD}$ is a 1GD inverse.
\end{proof}

\begin{thm}\label{T: 1GD Orden Charact}
Let $\f, \g \in \ed_k(V)$ be two finite potent endomorphisms. Then, $\f \, \leq^{1GD} \, \g$ if and only if \begin{itemize}
\item[I.)]$\Img(\g - \f)\subseteq U_{\f},$
\item[II.)] $\f \, \leq^{-} \, \g$ (Definition \ref{D: Minus FP}).
\end{itemize}
\end{thm}
\begin{proof}
If $\f \, \leq^{1GD} \, \g$ then we know that, in particular, $\f^{1GD}\circ \g =\f^{1GD}\circ \f$ for certain $\f^{1GD}\in X_{\f}(1GD)$ and then $\f^{m+1}=\f^m \circ \g $ by Lemma \ref{L: fm+1 fmg}. Hence, statement $I$ follows from Lemma \ref{L: Resta en Uf}. The second condition is exactly Corollary \ref{C: 1GD implica minus y space}. Conversely, let us suppose that both conditions on the statement hold and let us show the existence of a 1GD inverse of $\f$ which satisfies that $\f \circ \f^{1GD}=\g \circ \f^{1GD}$ and $\f^{1GD}\circ \f=\f^{1GD}\circ \g .$ Firstly, as $\f \, \leq^{-}\, \g,$ we know that 
\begin{equation}\label{eq: minus en 1GD}
\f \circ \f^-=\g\circ \f^- \text{ and } \f^-\circ \f=\f^-\circ \g,
\end{equation} for certain $\f^-\in X_{\f}(1).$
Moreover, as $\Img(\g - \f)\subseteq U_{\f},$ using Lemma \ref{L: Resta en Uf} one has that $\f^m\circ \f = \f^m \circ \g .$ Hence, we can ensure that the composition $\gamma = \f^-\circ \f \circ \g^{GD},$ for the $\f^-\in X_{\f}(1)$ satisfying \eqref{eq: minus en 1GD} and certain $\g^{GD}\in X_{\f}(GD)$ is a 1GD inverse of $\f$ in virtue of Lemma \ref{L: fmenos f g gdraz}. We shall check that $\f \circ \gamma = \g \circ \gamma$ and $\gamma \circ \f = \gamma \circ \g .$
To wit, \begin{align*}
\f \circ \gamma &=\f\circ \f^-\circ \f \circ \g^{GD}=\g \circ \f^-\circ \f \circ \g^{GD}=\g \circ \gamma\\
\gamma \circ \f & =\f^-\circ \f \circ \g^{GD}\circ \f=\f^-\circ \f,\\
\gamma \circ \g & =\f^-\circ \f \circ \g^{GD}\circ \g=\f^-\circ \g \circ \g^{GD}\circ \g=\f^-\circ \g,
\end{align*} were $\f\circ \g^{GD}\circ \f =\f$ because $X_{\g}(GD)\subseteq X_{\g}(1)\subseteq X_{\f}(1)$ in these conditions and the equalities are deduced from \eqref{eq: minus en 1GD} so the claim is proved.
\end{proof}

\begin{thm}\label{T: 1GD Orden Parcial}
The binary relation ``$\, \leq^{1GD} \,$'' is a partial order on the set of finite potent endomorphisms over an arbitrary $k-$vector space.
\end{thm}
\begin{proof}
Reflexivity holds directly. In order to check antisymmetry, let us consider two finite potent endomorphisms $\f , \g \in \ed_k(V)$ such that $\f \, \leq^{1GD} \, \g$ and $\g \, \leq^{1GD} \, \f .$ We shall use that $\f \circ \f^{1GD}=\f^{1GD}\circ \g$ for certain $\f^{1GD}\in X_{\f}(1GD)$ as $\f \, \leq^{1GD} \, \g$ that $\f=\f\circ \f^{1GD}\circ \g =\g \circ \f^{1GD}\circ \f$ for certain $\f^{1GD}\in X_{\f}(1GD)$ and that $\g=\g\circ \g^{1GD}\circ \f =\f \circ \g^{1GD}\circ \g$ for some $\g^{1GD}\in X_{\g}(1GD)$ in virtue of Theorem \ref{T: 3 equiv 1GD}. In these conditions, one has that \begin{align*}
\f & = \f \circ \f^{1GD}\circ \g=\f\circ \f^{1GD}\circ \g \circ \g^{1GD}\circ \g=\f \circ \f^{1GD}\circ \f \circ \g^{1GD}\circ \g=\\& =\f \circ \g^{1GD}\circ \g= \g
\end{align*} and antisymmetry is deduced. Let us finish the proof by showing that this relation is transitive. In order to do so, let us consider three finite potent endomorphisms $\f , \g ,\gamma \in \ed_k(V)$ such that $\f \, \leq^{1GD}\, \g$ and $\g \, \leq^{1GD}\, \gamma .$ Using Theorem \ref{T: 1GD Orden Charact} we know that $\f \, \leq^{-} \, \g$ and $\g \, \leq^{-}\, \gamma $ as well as $\Img(\g-\f)\subseteq U_{\f}$ and $\Img(\gamma -\g)\subseteq U_{\g}.$ Applying Corollary \ref{C: Ug en Uf} we know that $U_{\g}\subseteq U_{\f}.$ Now, using Lemma \ref{L: Resta en Uf}, we get, that $\g= \f +\phi$ and $\gamma = \g + \beta$ with $\Img(\phi)\subseteq U_{\f}$ and $\Img(\beta)\subseteq U_{\g}.$ By the previous observation, we deduce that $\Img(\beta)\subseteq U_{\g}\subseteq U_{\f}.$ Hence: $$\f^m\circ \gamma =\f^m\circ (\f+ \phi + \beta)=\f^{m+1} $$ because $\f^m\circ \phi =0=\f^m\circ \beta .$ Again, by Lemma \ref{L: Resta en Uf}, we get that $\Img(\gamma - \f)\subseteq U_{\f}.$ Hence, we have concluded as $\f \, \leq^- \, \gamma$ (the relation ``$\leq^{-}$'' is a partial order and, in particular, transitive, Theorem \ref{T: Minus is PO FP}) and $\Img(\gamma - \f)\subseteq U_{\f}.$ Using Theorem \ref{T: 1GD Orden Charact} in the other direction we get that $\f \, \leq^{1GD}\, \gamma$ as we wanted to see.
\end{proof}

\begin{exam}\label{Ex: Non Trivial}
We shall now offer an example showing that the relations ``$\, \leq^{GD1} \,$'' and ``$\leq^{1GD}$'' are not trivial outside the set of finite potent endomorphisms of index lesser or equal than one. In order to do so, let us consider the following two endomorphisms $\f , \g \colon \R^5 \to \R^5$ whose associated $(5\times 5)$ matrices with entries in the real numbers in the standard basis of $\R^5$ denoted as $\{ e_1,e_2, e_3 , e_4, e_5 \}$ are $$\f\equiv A=\left(\begin{matrix}
9 & 0 & 1 & 2 & -8 \\
-32 & 1 & -4 & -13 & 32 \\
2 & 0 & 0 & 1 & -2 \\
0 & 0 & 0 & 0 & 0 \\
9 & 0 & 1 & 2 & -8
\end{matrix}\right) \text{ and } \g \equiv B=\left(\begin{matrix}
9 & 0 & 1 & 2 & -8 \\
-37 & 1 & -4 & -13 & 37 \\
5 & 0 & 0 & 1 & -5 \\
-1 & 0 & 0 & 0 & 1 \\
9 & 0 & 1 & 2 & -8
\end{matrix}\right).$$ Moreover, let us consider the following Jordan forms of both, which are precisely their respective core-nilpotent decompositions ($A=P\cdot \left(\begin{matrix}
C & 0 \\
0 & N
\end{matrix}\right) \cdot P^{-1}$): $$A=P\cdot \left(\begin{matrix}
1 & 0 & 0 & 0 & 0 \\
0 & 1 & 0 & 0 & 0 \\
0 & 0 & 0 & 1 & 0 \\
0 & 0 & 0 & 0 & 0 \\
0 & 0 & 0 & 0 & 0
\end{matrix}\right) \cdot P^{-1} \text{ and } B=P\cdot \left(\begin{matrix}
1 & 0 & 0 & 0 & 0 \\
0 & 1 & 0 & 0 & 0 \\
0 & 0 & 0 & 1 & 0 \\
0 & 0 & 0 & 0 & 1 \\
0 & 0 & 0 & 0 & 0
\end{matrix}\right) \cdot P^{-1}, $$ for $$P=\left(\begin{matrix}
1 & 2 & -1 & 0 & 3 \\
0 & 1 & 4 & 5 & -2 \\
0 & 0 & 1 & -3 & 1 \\
0 & 0 & 0 & 1 & 2 \\
1 & 2 & -1 & 0 & 4
\end{matrix}\right). $$ Note that $i(\f)=2 ,$ $i(\g)=3 ,$ $$U_{\f}=\Ker(\f^2)=\Ker(\f^3)=\langle e_3, e_4, e_5 \rangle = U_{\g}=\Ker(\g^3)=\Ker(\g^4).$$   In fact, $$W_{\f}=\Img(\f^2)=\Img(\f^3)=\langle e_1, e_2 \rangle =W_{\g}=\Img(\g^3)=\Img(\g^4), $$ and $\f_{\vert_{W_{\f}}}\equiv C_A=\left(\begin{matrix}
1 & 0 \\
0 & 1
\end{matrix}\right)=C_B\equiv \g_{\vert_{W{_\f}}}.$ It is a straightforward calculation that $$\Img(\g-\f)=\langle e_4 \rangle\subseteq U_{\f}= \langle e_3, e_4, e_5 \rangle .$$ 
 Moreover, let us consider the following $\{1\}-$inverse of $A:$ $$A^-=P\cdot \left(\begin{matrix}
1 & 0 & 0 & 0 & 0 \\
0 & 1 & 0 & 0 & 0 \\
0 & 0 & 0 & 0 & 0 \\
0 & 0 & 1 & 0 & 0 \\
0 & 0 & 0 & 0 & 0
\end{matrix}\right) \cdot P^{-1}= \left(\begin{matrix}
11 & 0 & 1 & 3 & -10 \\
-5 & 1 & 1 & -2 & 5 \\
-21 & 0 & -3 & -9 & 21 \\
7 & 0 & 1 & 3 & -7 \\
11 & 0 & 1 & 3 & -10
\end{matrix}\right). $$ Readers can check that $A\cdot A^-=B\cdot A^-$ and $A^-\cdot A=A^-\cdot B ,$ this is, $A\, \leq^- \, B.$ Hence, using Theorem \ref{T: Charact GD1 Order FP} and Theorem \ref{T: 1GD Orden Charact} we conclude that $A \, \leq^{GD1} \, B$ and $A\, \leq^{1GD} \, B.$ We conclude that there are non trivial examples of matrices ordered for these relations which do not have index 1.
\end{exam}

\subsection{On other binary relations with GD1 and 1GD inverses}\label{ss:New Bin Rel}

This last section is devoted to define two binary relations using several compositions of GD1 and 1GD inverses and to show that they are partial orders. Indeed, we will prove that these relations are equivalent to two well known partial orders, the minus partial order and the G-Drazin partial order.

\begin{prop}\label{P: CompGD11GD 1GDGD1}
Let us consider a finite potent endomorphism $\f \in \ed_k(V).$ Then, both maps of sets
$$\begin{array}{rccl}
\tilde{\Gamma}^{\f} \colon & X_{\f}(GD1)\times X_{\f}(1GD)  & \rightarrow & X_{\f}(GD)\cap X_{\f}(1,2) \smallskip\\
& (\f^{GD1},\tilde{\f}^{1GD}) & \mapsto & \f^{GD1}\circ \f \circ \tilde{\f}^{1GD}.
\end{array} $$ 
and 
$$\begin{array}{rccl}
\Gamma^{+,\f} \colon & X_{\f}(1GD)\times X_{\f}(GD1)  & \rightarrow & X_{\f}(1,2) \smallskip\\
& (\tilde{\f}^{1GD},\f^{GD1}) & \mapsto & \tilde{\f}^{1GD}\circ \f \circ \f^{GD1}.
\end{array} $$
are surjective.
\end{prop}
\begin{proof}
Let us suppose that $\f^{GD1}=\f^{GD}\circ \f \circ \f^{-},$  $\tilde{\f}^{1GD}=\tilde{\f}^-\circ \f \circ \tilde{\f}^{GD}$ for certain $\f^{GD}, \tilde{\f}^{GD}\in X_{\f}(GD)$ and $\f^-,\tilde{\f}^-\in X_{\f}(1).$ Let us study first the map of sets $\tilde{\Gamma}^{\f}.$ The map is well defined. Clearly, 

\begin{align*}
\f^{GD1}\circ \f \circ \tilde{\f}^{1GD}& =(\f^{GD}\circ \f \circ \f^{-})\circ \f \circ (\tilde{\f}^-\circ \f \circ \tilde{\f}^{GD})=\f^{GD}\circ \f \circ \tilde{\f}^-\circ \f \circ \tilde{\f}^{GD}=\\ & =\f^{GD}\circ \f \circ \tilde{\f}^{GD};
\end{align*} 
which is a G-Drazin inverse in virtue of Lemma \ref{P: CompG-Drazin}. Indeed, 
\begin{align*}
(\f^{GD1} \circ \f \circ \tilde{\f}^{1GD})& \circ \f \circ (\f^{GD1}\circ \f \circ \tilde{\f}^{1GD})=(\f^{GD}\circ \f \circ \tilde{\f}^{GD})\circ \f \circ (\f^{GD}\circ \f \circ \tilde{\f}^{GD})= \\ & =\f^{GD}\circ \f \circ \f^{GD}\circ \f\circ \tilde{\f}^{GD}=\f^{GD}\circ \f \circ \tilde{\f}^{GD}.
\end{align*}
Hence, $\tilde{\Gamma}^{\f}(\tilde{\f}^{1GD},\f^{GD1})$ is a reflexive generalized inverse and a G-Drazin inverse. Moreover, if we consider a $\f^+\in X_{\f}(GD)\cap X_{\f}(1,2),$ then, by Proposition \ref{T: Charact Algebraica} and Proposition \ref{P: Charact Algebraica 1GD}, we have that $\f^+\circ \f \circ \f^+ \in X_{\f}(GD1)\cap X_{\f}(1GD).$ In fact, one has that $$\tilde{\Gamma}^{\f}(\f^+,\f^+)=\f^+\circ \f \circ \f^+=\f^+ $$
so we conclude that $\tilde{\Gamma}^{\f}$ is surjective. Let us now study the map $\Gamma^{+,\f}.$ Firstly, note that $$\Gamma^{+,\f}(\tilde{\f}^{1GD},\f^{GD1})=\tilde{\f}^{1GD}\circ \f \circ \f^{GD1}=(\tilde{\f}^-\circ \f\circ \tilde{\f}^{GD})\circ \f \circ (\f^{GD}\circ \f \circ \f^-)=\tilde{\f}^-\circ \f \circ \f^-. $$ It is well defined. \begin{align*}
(\tilde{\f}^{1GD}\circ \f \circ \f^{GD1})\circ \f \circ (\tilde{\f}^{1GD}\circ \f \circ \f^{GD1})& = (\tilde{\f}^-\circ \f \circ \f^-)\circ \f \circ (\tilde{\f}^-\circ \f \circ \f^-)=\\ &=\tilde{\f}^-\circ \f \circ \f^-.
\end{align*}
For the surjectivity, let us consider any $\hat{\f}^+\in X_{\f}(1,2).$ Then, in virtue of Proposition \ref{T: Charact Algebraica} and Proposition \ref{P: Charact Algebraica 1GD}, it is clear that $\f^{GD}\circ \f \circ \hat{\f}^{+}\in X_{\f}(GD1)$ and that $\hat{\f}^+\circ \f \circ \tilde{\f}^{GD}\in X_{\f}(1GD)$ for whatever the G-Drazin inverses $\f^{GD},\tilde{\f}^{GD}\in X_{\f}(GD)$ considered. Therefore: \begin{align*}
(\hat{\f}^+\circ \f \circ \tilde{\f}^{GD})\circ \f \circ (\f^{GD}\circ \f \circ \hat{\f}^{+})& = \hat{\f}^+\circ \f \circ \f^{GD}\circ \f \circ \hat{\f}^{+}= \\& =\hat{\f}^+\circ \f\circ \hat{\f}^+=\hat{\f}^+
\end{align*}
and thus we conclude the proof.
\end{proof}

\begin{defn}\label{D: GD11GD bin rel}
Let $\f, \g \in \ed_k(V)$ be two finite potent endomorphisms. We will say that $\f$ is below $\g$ under the GD1-1GD relation, and it will be denoted as $\f \, \leq^{GD1-1GD} \, \g ,$ when there exists some elements $\tilde{\f}_1 , \tilde{\f}_2 \in \Img(\tilde{\Gamma}^{\f})$ (Proposition \ref{P: CompGD11GD 1GDGD1}) such that \begin{align*}
\f\circ \tilde{\f}_1=\g\circ \tilde{\f}_1 \\
\tilde{\f}_2 \circ \f=\tilde{\f}_2 \circ \g.
\end{align*} 
\end{defn}

\begin{defn}\label{D: 1GDGD1 bin rel}
Let $\f, \g \in \ed_k(V)$ be two finite potent endomorphisms. We will say that $\f$ is below $\g$ under the GD1-1GD relation, and it will be denoted as $\f \, \leq^{1GD-GD1} \, \g ,$ when there exists some elements $\f^{+}_1,\f^{+}_2 \in \Img(\Gamma^{+,\f})$ (Proposition \ref{P: CompGD11GD 1GDGD1}) such that \begin{align*}
\f\circ \f^{+}_1=\g\circ \f^{+}_1 \\
\f^{+}_2 \circ \f=\f^{+}_2 \circ \g.
\end{align*} 
\end{defn}

\begin{lem}\label{L: GD11GD et al son G-Based}
Let us consider two finite potent endomorphisms $\f,\g \in \ed_k(V).$ Then \begin{itemize}
\item $\f \, \leq^{GD1-1GD} \, \g ,$ (in the sense of Definition \ref{D: GD11GD bin rel}) if and only if there exists an element $\tilde{\f}\in \Img(\tilde{\Gamma}^{\f})$ (Proposition \ref{P: CompGD11GD 1GDGD1}) such that 
\begin{align*}
\f\circ \tilde{\f}& =\g\circ \tilde{\f}\\
\tilde{\f} \circ \f & =\tilde{\f} \circ \g .
\end{align*}
\item $\f \, \leq^{1GD-GD1} \, \g ,$ (in the sense of Definition \ref{D: 1GDGD1 bin rel}) if and only if there exists an element $\f^+\in \Img(\Gamma^{+,\f})$ (Proposition \ref{P: CompGD11GD 1GDGD1}) such that \begin{align*}
\f\circ \f^{+} & =\g\circ \f^{+} \\
\f^{+} \circ \f & =\f^{+} \circ \g .
\end{align*} 

\end{itemize}
\end{lem}
\begin{proof}
It is a direct consequence of the surjectivity of both maps of sets presented in Proposition \ref{P: CompGD11GD 1GDGD1}.
\end{proof}

\begin{rem}
Notice that Lemma \ref{L: GD11GD et al son G-Based} states that the binary relations ``$\leq^{GD1-1GD} $'' and ``$\leq^{1GD-GD1}$'' are precisely $\mathcal{G}-$based orders (Definition \ref{D: G-based rel}).
\end{rem}

\begin{thm}\label{T: Bilateral Relations are GDraz and Minus}
Let $\f , \g \in \ed_k(V)$ be two finite potent endomorphisms. Then: \begin{itemize}
\item[I.)] $\f \, \leq^{GD1-1GD} \, \g ,$ if and only if $\f \, \leq^{GD}\, \g$ (Definition \ref{D: G-Drazin order fp}).

\item[II.)] $\f \, \leq^{1GD-GD1} \, \g ,$ if and only if $\f \, \leq^- \, \g$ (Definition \ref{D: Minus FP}).
\end{itemize}
\end{thm}
\begin{proof}
Let us start by proving $I.$ If  $\f \, \leq^{GD1-1GD} \, \g ,$ then, in particular, the $\tilde{\f}\in \Img(\tilde{\Gamma}^{\f})$ is a G-Drazin inverse so Definition \ref{D: G-Drazin order fp} is satisfied trivially. Conversely, let us suppose that $\f \, \leq^{GD}\, \g$ and therefore $\f \circ \f^{GD}=\g \circ \f^{GD}$ and $\f^{GD}\circ \f=\f^{GD}\circ \g$ for certain $\f^{GD}\in X_{\f}(GD).$ Now, take the endomorphism $\tilde{\f}=\f^{GD}\circ \f \circ \f^{GD}$ which is a G-Drazin inverse by Proposition \ref{P: CompG-Drazin} and a reflexive generalized inverse in virtue of Lemma \ref{L: Gamma Reflex}. Hence the constructed $\tilde{\f}\in \Img(\tilde{\Gamma}^{\f})$ and clearly $\f \tilde{\f}=\g \circ \tilde{\f}$ and $\tilde{\f}\circ \f=\tilde{\f}\circ \g$ so $\f \, \leq^{GD1-1GD} \, \g .$\\Let us now prove $II.$ If $\f \, \leq^{1GD-GD1} \, \g ,$ then, as we know that any element belonging to $\Img(\Gamma^{+,\f})$ is, in particular, a $\{1\}-$inverse of $\f ,$ then $\f \, \leq^{-}\, \g$ in virtue of the characterization given in Theorem \ref{T: Equiv Minus FP}. Conversely, if $\f \, \leq^{-}\, \g ,$ then, for some $\f^-\in X_{\f}(1)$ we know that  $\f \circ \f^-=\g\circ \f^-$ and $\f^-\circ \f=\f^-\circ \g .$ Hence, for this same $\f^-,$ consider the composition $\f^+=\f^-\circ \f\circ \f^+.$ By Lemma \ref{L: Gamma Reflex} we know $\f^+ \in X_{\f}(1,2)=\Img(\Gamma^{+,\f}).$ Again, it is direct that $\f \circ \f^+=\g \circ \f^+$ and $\f^+\circ \f=\f^+\circ \g$ (using the relation given by $\f \, \leq^- \, \g$ and the constructed $\f^+$). Therefore, we conclude that $\f \, \leq^{1GD-GD1} \, \g $ and the claim is proved.
\end{proof}

\begin{cor}
The binary relations ``$\leq^{GD1-1GD}$'' and ``$\leq^{1GD-GD1}$'' are partial orders in the set of finite potent endomorphisms.
\end{cor}
\begin{proof}
It is deduced directly from the equivalences presented in Theorem \ref{T: Bilateral Relations are GDraz and Minus} and Theorem \ref{T: Minus is PO FP} and Definition \ref{D: G-Drazin order fp}.
\end{proof}

\section*{Declarations}

\begin{itemize}
\item Funding: This work was supported by {\it Agencia Estatal de Investigación} (Spain) through grant PID2023-151823NB-I00. 
\item The author has no relevant financial or non-financial interests to disclose.
\item Conflict of interest/Competing interests: none.
\item Ethics approval: not-applicable.
\item Consent to participate: not-applicable.
\item Consent for publication: the author authorises the publication of the previous notes.
\item Availability of data and materials: Not-applicable.
\item Code availability: not-applicable.
\item Authors' contributions: Not-applicable.
\end{itemize}


\medskip











 



\begin{thebibliography}{MMMM}

\bibitem{RendDaa} Alba Alonso, D.; \textit{On some pre-orders and partial orders of linear operators on infinite dimensional vector spaces}, Rend. Circ. Mat. Palermo, II. Ser, 74, 2349-2382, (2024).


\bibitem{MinusDaa} Alba Alonso, D.; \textit{On the Characterization of Finite Potent Endomorphisms via their Finite Potent 1-inverses and a Generalization of the Minus Partial Order}, Results Math 80, 150 (2025). https://doi.org/10.1007/s00025-025-02469-4


\bibitem{Die-Fpa} Alba Alonso D., Pablos Romo F.; \textit{On the characterization of the set of reflexive generalized inverses of finite potent endomorphisms}, Filomat 38(14),  4955–4972,  (2024).

\bibitem{AST}  Argerami, M., Szechtman, F., Tifenbach, R.; \textit{On Tate's trace}, Linear Multilinear Algebra 55(6), 515-520, (2007).

\bibitem{MPFP} Cabezas Sánchez V., Pablos Romo F.; \textit{Moore-Penrose Inverse of Some Linear Maps on Infinite-Dimensional Vector Spaces}, Electron. J. Linear Algebra 36, 570-586, (2020).


\bibitem{GD1N} Maharana, G., Sahoo, JK., Thome, N.; \textit{G-Drazin inverse combined with inner inverse}, Linear Multilinear Algebra, 73(1), 106–121, (2024). https://doi.org/10.1080/03081087.2024.2316786



\bibitem{Ind}  Mitra, S.K., Bhimasankaram, P.,Malik, S.B.; \textit{Matrix Partial Orders, Shorted Operators and Applications}, World Scientific, (2010).

\bibitem{GMos} Mosic, D.; \textit{Weighted G-Drazin inverse for operators on Banach spaces} Carpathian J. Math., 35,2, 171-184, (2019).

 
\bibitem{Pa} Pablos Romo, F. \textit{Classification of finite potent endomorphisms}, Linear Algebra Appl., 440, 266-277, (2014).

\bibitem{Pa-CN} Pablos Romo, F.; \textit{Core-Nilpotent Decomposition and new generalized inverses  of  Finite Potent Endomorphisms}, Linear Multilinear Algebra 68(11), 2254-2275, (2020).

\bibitem{Fpa-CN}  Pablos Romo, F.; \textit{Core-Nilpotent Decomposition of Infinite Dimensional Vector Spaces}, Mediterr. J. Math., (2021).

\bibitem{DraFP} Pablos Romo F.; \textit{On the Drazin Inverse of Finite Potent Endomorphisms}, Linear Multilinear Algebra, 67(10), 627-647,  (2019).

\bibitem{PaAl} Pablos Romo, F.; Alba Alonso, D.; \textit{On the Explicit Computation of \\ the set of 1-inverses of a square matrix}, Linear Multilinear Algebra 71(18), 2869-2876, (2023).

\bibitem{FPGD}  Pablos Romo F; \textit{On G-Drazin inverses of finite potent endomorphisms and arbitrary square matrices}, Linear Multilinear Algebra, 70(12), 2227-2247, (2019).


\bibitem{sahoo} Sahoo, J.K., Boggarapu, P., Behera, R. Thome, N.; \textit{GD1 inverse and 1GD inverse for bounded operators on Banach spaces}, Comp. Appl. Math. 42,212, (2023). https://doi.org/10.1007/s40314-023-02355-1

\bibitem{Ta}  Tate, J. \textit{Residues of Differentials on Curves}, Ann. Scient. \'Ec.
Norm. Sup. 1,4a s\'erie, 149-159,(1968).

\bibitem{GDraz} Wang, H., Liu, X.; \textit{Partial orders based on core-nilpotent decomposition}, Linear Algebra Appl., 488, 235-248, (2016).



\end{thebibliography}
\end{document}